\documentclass[12pt]{amsart}

\usepackage{amssymb, enumerate, tikz}

\newtheorem{theorem}{Theorem}[section]
\newtheorem{lemma}[theorem]{Lemma}
\newtheorem{sublemma}{}[theorem]
\newtheorem{corollary}[theorem]{Corollary}
\newtheorem{prop}[theorem]{Proposition}

\theoremstyle{definition}
\newtheorem{dfn}[theorem]{Definition}

\newcommand{\mcal}[1]{\ensuremath{\mathcal{#1}}}

\newcommand{\del}{\backslash}
\newcommand{\iso}{\cong}
\newcommand{\la}{\lambda}
\newcommand{\dash}{\nobreakdash-\hspace{0pt}}

\DeclareMathOperator{\si}{si}
\DeclareMathOperator{\co}{co}
\DeclareMathOperator{\cl}{cl}
\DeclareMathOperator{\ra}{r}

\begin{document}

\title[Contracting cocircuit elements]
{Contracting an element from a cocircuit.}

\author{Rhiannon Hall}

\address{School of Information Systems, Computing and Mathematics,
Brunel University, Uxbridge UB8 3PH, United Kingdom}

\email{rhiannon.hall@brunel.ac.uk}

\thanks{The research of the first author was supported by a
Nuffield Foundation Award for Newly Appointed Lecturers in Science,
Engineering and Mathematics}

\author{Dillon Mayhew}
\address{School of Mathematics, Statistics and Computer Science,
Victoria University of Wellington, P.O. BOX 600,
Wellington, New Zealand}

\email{dillon.mayhew@mcs.vuw.ac.nz}

\thanks{The research of the second author was supported by
a NZ Science \& Technology Postdoctoral Fellowship.}

\date{\today}



\begin{abstract}
We consider the situation that $M$ and $N$ are $3$\dash connected
matroids such that $|E(N)| \geq 4$ and $C^{*}$ is a cocircuit of $M$
with the property that $M / x_{0}$ has an $N$\dash minor for some
$x_{0} \in C^{*}$.
We show that either there is an element $x \in C^{*}$ such that
$\si(M / x)$ or $\co(\si(M / x))$ is $3$\dash connected
with an $N$\dash minor, or there is a four-element fan of $M$
that contains two elements of $C^{*}$ and an element $x$ such that
$\si(M / x)$ is $3$\dash connected with an $N$\dash minor.
\end{abstract}

\maketitle

\section{Introduction}
\label{intro}

There are a number of tools in matroid theory that tell us when
we can remove an element or elements from a matroid, while
maintaining both the presence of a minor and a certain type
of connectivity.
Some recent results are of this type, but have the
additional restriction that the element(s) must have a certain
relation to a given substructure in the matroid.
For example, Oxley, Semple, and Whittle~\cite{basis}, consider
a given basis of a matroid and consider either contracting elements
that are in the basis, or deleting elements that are not in
the basis.
Hall~\cite{hyperplane} has investigated when it is possible to
contract an element from a given hyperplane in a $3$\dash connected
matroid and remain $3$\dash connected (up to parallel pairs).

We make a contribution to this collection of tools by investigating
the circumstances under which we can contract an element from
a cocircuit while maintaining both the presence of a minor
and $3$\dash connectivity (up to parallel pairs), and the
structures which prevent us from doing so.
Our result has been employed by Geelen, Gerards, and
Whittle~\cite{GGW} in their characterization of when three
elements in a matroid lie in a common circuit.

\begin{theorem}
\label{thm1}
Suppose that $M$ and $N$ are $3$\dash connected matroids such that
$|E(N)| \geq 4$ and $C^{*}$ is a cocircuit of $M$ with the property
that $M / x_{0}$ has an $N$\dash minor for some $x_{0} \in C^{*}$.
Then either:
\begin{enumerate}[(i)]
\item there is an element $x \in C^{*}$ such that $\si(M / x)$ is
$3$\dash connected and has an $N$\dash minor;
\item there is an element $x \in C^{*}$ such that $\co(\si(M / x))$ is
$3$\dash connected and has an $N$\dash minor; or,
\item there is a sequence of elements $(x_{1},\, x_{2},\, x_{3},\, x_{4})$
from $E(M)$ such that $\{x_{1},\, x_{2},\, x_{3}\}$ is a
circuit, $\{x_{2},\, x_{3},\, x_{4}\}$ is a cocircuit,
$x_{1},\, x_{3} \in C^{*}$, and $\si(M / x_{2})$ is $3$\dash connected
with an $N$\dash minor.
\end{enumerate}
\end{theorem}

The next example shows that statement~(ii) of
Theorem~\ref{thm1} is necessary.

\begin{figure}[htb]

\centering

\begin{tikzpicture}

\path (0,0) node(v0){};
\path (25:2.5cm) node(v1){};

\draw[thick] (v0.center) -- (v1.center) node[pos=0.5](v2){};

\path (70:2.5cm) node(v3){};

\draw[thick] (v0.center) -- (v3.center) node[pos=0.5](v4){};

\draw[thick] (v3.center) -- (v2.center);
\draw[thick] (v4.center) -- (v1.center);

\path (intersection of v3--v2 and v4--v1) node(v5){};

\path (0:2.5cm) node(v6){};

\draw[thick] (v0.center) -- (v6.center) node[pos=0.5](v7){};

\draw[thick] (v6.center) -- (v2.center);
\draw[thick] (v7.center) -- (v1.center);

\path (intersection of v6--v2 and v7--v1) node(v8){};

\foreach \x in {0,...,8}
\filldraw (v\x) circle (0.06cm);

\path (v3) node[left]{$a$};
\path (v4) node[left]{$b$};
\path (v7) +(0,-0.3) node[anchor=mid]{$c$};
\path (v6) +(0,-0.3) node[anchor=mid]{$d$};

\end{tikzpicture}
  
\caption{The graphic matroid $M(K_{5} \del e)$.}

\label{fig2}

\end{figure}
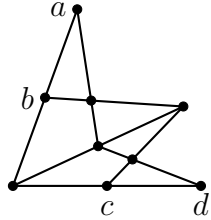

Consider the rank\dash $4$ matroid $M$ whose geometric
representation is shown in Figure~\ref{fig2}.
Note that $M \iso M(K_{5} \del e)$.
The set $C = \{a,\, b,\, c,\, d\}$ is a circuit of $M$, and
hence a cocircuit of $M^{*}$.
Moreover $M^{*} / x$ has a minor isomorphic to $M(K_{4})$
for any element $x \in C$.
However $\co(M \del x)$ is not $3$\dash connected,
as it contains a parallel pair, so $\si(M^{*} / x)$ is
not $3$\dash connected.
On the other hand $\co(\si(M^{*} / x))$ is $3$\dash connected,
and has a minor isomorphic to $M(K_{4})$.

More generally we suppose that $r$ is an integer greater than two.
Consider a basis $A = \{a_{1},\ldots, a_{r}\}$ in the projective space
$\mathrm{PG}(r - 1,\, \mathbb{R})$.
Let $l$ be a line of $\mathrm{PG}(r - 1,\, \mathbb{R})$ that is
freely placed relative to $A$, and for all $i \in \{1,\ldots, r\}$
let $b_{i}$ be the point that is in both $l$ and the hyperplane of
$\mathrm{PG}(r - 1,\, \mathbb{R})$ spanned by
$A - a_{i}$.
Let $B = \{b_{1},\ldots, b_{r}\}$.
We will use $\Theta_{r}$ to denote the restriction of
$\mathrm{PG}(r - 1,\, \mathbb{R})$ to $A \cup B$.

Suppose that $\Theta_{r}'$ is an isomorphic copy of $\Theta_{r}$
with $\{a_{1}',\ldots, a_{r}'\} \cup B$ as its ground set.
Assume also that the isomorphism from $\Theta_{r}$ to
$\Theta_{r}'$ acts as the identity on $B$ and takes $a_{i}$ to
$a_{i}'$ for all $i \in \{1,\ldots, r\}$.
Let $M$ be the generalized parallel connection of
$\Theta_{r}$ and $\Theta_{r}'$.
That is, $M$ is a matroid on the ground set $A \cup A' \cup B$
and the flats of $M$ are exactly the sets $F$ such that
$F \cap (A \cup B)$ is a flat of $\Theta_{r}$ and
$F \cap (A' \cup B)$ is a flat of $\Theta_{r}'$.
Note that if $r = 3$ then $M$ is isomorphic to $M(K_{5} \del e)$,
the matroid illustrated in Figure~\ref{fig2}.

It is easy to see that $\Theta_{r}$ is self-dual and that
$C = (A - a_{1}) \cup (A' - a_{1}')$ is a circuit of $M$, and hence a
cocircuit of $M^{*}$.
Moreover $M^{*} / x$ has an isomorphic copy of $\Theta_{r}$ as a
minor for every element $x \in C$.
We note that every three-element subset of $A$ is a circuit
of $M^{*}$.
Thus $A - x$ is a parallel class of $M^{*} / x$ for every $x \in C \cap A$.
However the simplification of $M^{*} / x$ contains a unique series pair,
and is therefore not $3$\dash connected.
On the other hand $\co(\si(M^{*} / x))$ is $3$\dash connected, and
has a minor isomorphic to $\Theta_{r}$.

The structure described in the last example has been discovered before.
The matroid $\Theta_{r}$ is a fundamental object in the
generalized $\Delta\textrm{-}Y$ operation of Oxley, Semple, and
Vertigan~\cite{oxley6}.
Furthermore this construction is an example of a `crocodile', as
described by Hall, Oxley, and Semple~\cite{sequences}.

To see that statement~(iii) of Theorem~\ref{thm1} is necessary
consider the graph $G$ shown in Figure~\ref{fig1}.
Let $C^{*}$ be the cocircuit of $M = M(G)$ comprising the
edges incident with the vertex $a$.
It is easy to see that if $x$ is any edge between $a$ and a vertex
in $\{b,\, c,\, d,\, e,\, f\}$ then $M / x$ has a minor isomorphic to
$M(K_{6})$, and that these are the only edges in $C^{*}$ with this property.
But in this case neither $\si(M / x)$ nor $\co(\si(M / x))$
is $3$\dash connected.
On the other hand, if we let $x_{1}$ be the edge $ad$, $x_{2}$
be $cd$, $x_{3}$ be $ac$, and $x_{4}$ be $bc$, then
$(x_{1},\, x_{2},\, x_{3},\, x_{4})$ is a sequence of the type described
in statement~(iii) of Theorem~\ref{thm1}.

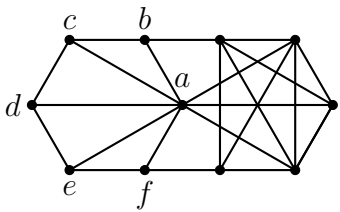
\begin{figure}[htb]

\centering

\begin{tikzpicture}

\begin{scope}
\path (0,0) node(v1){};
\foreach \x/\xtext in {60/2,120/3,180/11,240/4,300/5}
\path ++(-1,0) +(\x:1cm) node(v\xtext){};
\foreach \x/\xtext in {0/8,60/9,120/10,240/6,300/7}
\path ++(1,0) +(\x:1cm) node(v\xtext){};
\end{scope}

\foreach \x in {2,3,4,5,7,8,9,11}
\draw[thick] (v1.center) -- (v\x.center);

\foreach \x in {6,7,8}
\draw[thick] (v10.center) -- (v\x.center);

\foreach \x in {6,7}
\draw[thick] (v9.center) -- (v\x.center);

\draw[thick] (v8.center) -- (v7.center);

\draw[thick] (v2.center) -- (v3.center) -- (v11.center) --
(v4.center) -- (v5.center) -- (v6.center) -- (v7.center) --
(v8.center) -- (v9.center) -- (v10.center) -- cycle;

\foreach \x in {1,...,11}
\filldraw (v\x.center) circle (0.06cm);

\path (v1) node[above=0.1cm]{$a$};
\path (v2) node[above]{$b$};
\path (v3) node[above]{$c$};
\path (v11) node[left]{$d$};
\path (v4) node[below]{$e$};
\path (v5) node[below]{$f$};

\end{tikzpicture}

\caption{The graph $G$.}

\label{fig1}

\end{figure}

Our main result shows that there are essentially only two
structures that prevent us from finding an element $x \in C^{*}$
such that $\si(M / x)$ is $3$\dash connected with an $N$\dash minor.
These structures are named `segment-cosegment pairs' and
`four-element fans'.
The dual of the matroid in Figure~\ref{fig2} contains a
segment-cosegment pair, and the graph in Figure~\ref{fig1}
contains a four-element fan.
Before describing our result in detail we fix some terminology.
Suppose that $M$ is a matroid.
Recall that a \emph{triangle} of $M$ is a three-element circuit,
and a \emph{triad} is a three-element cocircuit.
A \emph{four-element fan} of $M$ is a sequence
$(x_{1},\, x_{2},\, x_{3},\, x_{4})$ of distinct elements from
$E(M)$ such that $\{x_{1},\, x_{2},\, x_{3}\}$ is a triangle and
$\{x_{2},\, x_{3},\, x_{4}\}$ is a triad.
A \emph{segment} of $M$ is a set $L$ such that $|L| \geq 3$ and
every three-element subset of $M$ is a triangle, and a \emph{cosegment}
of $M$ is a segment of $M^{*}$.
We say that $(L,\, L^{*})$ is a \emph{segment-cosegment pair} if
$L = \{x_{1},\ldots, x_{t}\}$ is a segment of $M$, and
$L^{*} = \{y_{1},\ldots, y_{t}\}$ is a set such that
$L \cap L^{*} = \emptyset$ and for every $x_{i} \in L$ the set
$(\cl(L) - x_{i}) \cup y_{i}$ is a cocircuit.
Segment-cosegment pairs will be considered in detail in
Section~\ref{segments}.
A \emph{spore} is a pair $(P,\, s)$ such that
$P$ is a rank\dash one flat, and $P \cup s$ is a cocircuit.
A matroid $M$ is \emph{$3$\dash connected up to a unique spore}
if $M$ contains a single spore $(P,\, s)$, and whenever
$(X,\, Y)$ is a $k$\dash separation of $M$ for some $k < 3$
then either $X \subseteq P \cup s$ or $Y \subseteq P \cup s$.
Theorem~\ref{thm1} follows from the next result.
It gives a more detailed analysis of the structures
we encounter.

\begin{theorem}
\label{main}
Suppose that $M$ and $N$ are $3$\dash connected matroids such that
$|E(N)| \geq 4$ and $C^{*}$ is a cocircuit of $M$ with the property
that $M / x_{0}$ has an $N$\dash minor for some $x_{0} \in C^{*}$.
Then either:
\begin{enumerate}[(i)]
\item there is an element $x \in C^{*}$ such that $\si(M / x)$ is
$3$\dash connected and has an $N$\dash minor;
\item there is a four-element fan $(x_{1},\, x_{2},\, x_{3},\, x_{4})$
of $M$ such that $x_{1},\, x_{3} \in C^{*}$, and $\si(M / x_{2})$ is
$3$\dash connected with an $N$\dash minor;
\item there is a segment-cosegment pair $(L,\, L^{*})$ such that
$L \subseteq C^{*}$, and $\cl(L) - L$ contains a single element $e$.
In this case $e \notin C^{*}$ and $\si(M / e)$ is $3$\dash connected
with an $N$\dash minor.
Moreover $M / \cl(L)$ is $3$\dash connected with an $N$\dash minor,
and if $x_{i} \in L$ then $M / x_{i}$ is $3$\dash connected up
to a unique spore $(\cl(L) - x_{i},\, y_{i})$; or,
\item there is a segment-cosegment pair $(L,\, L^{*})$ such that
$L$ is a flat and $|L - C^{*}| \leq 1$.
In this case $M / L$ is $3$\dash connected with an $N$\dash minor,
and if $x_{i} \in L$ then $M / x_{i}$ is $3$\dash connected up to
a unique spore $(L - x_{i},\, y_{i})$.
\end{enumerate}
\end{theorem}

We note that if $(L,\, L^{*})$ is a segment-cosegment
pair of the matroid $M$, and $M / \cl(L)$ has an $N$\dash minor,
then $|E(M) - \cl(L)| \geq 4$.
Under these hypotheses Proposition~\ref{prop4} tells us that
$M / \cl(L)$ is isomorphic to $\co(\si(M / x_{i}))$ for any element
$x_{i} \in L$.
Therefore Theorem~\ref{thm1} does indeed follow from
Theorem~\ref{main}.

By dualizing we immediately obtain the following corollary of
Theorem~\ref{thm1}.

\begin{theorem}
\label{thm2}
Suppose that $M$ and $N$ are $3$\dash connected matroids such that
$|E(N)| \geq 4$ and $C$ is a circuit of $M$ with the property
that $M \del x_{0}$ has an $N$\dash minor for some $x_{0} \in C$.
Then either:
\begin{enumerate}[(i)]
\item there is an element $x \in C$ such that $\co(M \del x)$ is
$3$\dash connected and has an $N$\dash minor;
\item there is an element $x \in C$ such that $\si(\co(M \del x))$ is
$3$\dash connected and has an $N$\dash minor; or,
\item there is a four-element fan $(x_{1},\, x_{2},\, x_{3},\, x_{4})$
in $M$ such that $x_{2},\, x_{4} \in C$, and $\co(M \del x_{3})$ is
$3$\dash connected with an $N$\dash minor.
\end{enumerate}
\end{theorem}

We note that Lemos~\cite{lemos2} has considered the situation that a
$3$\dash connected matroid $M$ contains a circuit $C$ with
the property that $M \del x$ is not $3$\dash connected
for any element $x \in C$.
He shows that in this case $C$ meets at least two triads of
$M$.

In Section~\ref{essentials} we introduce essential notions of
matroid connectivity.
Section~\ref{segments} contains a detailed discussion of one
of the structures we uncover: segment-cosegment pairs.
In Section~\ref{prelim} we collect some preliminary lemmas, and in
Section~\ref{mainproof} we complete the proof of Theorem~\ref{main}.
Notation and terminology generally follows that of Oxley~\cite{oxley},
except that the simple (respectively cosimple) matroid associated with
the matroid $M$ is denoted $\si(M)$ (respectively $\co(M)$).
We consistently write $z$ instead of $\{z\}$ for the set containing
the single element $z$.

\section{Essentials}
\label{essentials}

This section collects some elementary results on matroid
connectivity.
Let $M$ be a matroid on the ground set $E$.
The \emph{connectivity function} of $M$, denoted by
$\la_{M}$ (or $\la$ when there is no ambiguity), takes subsets of
$E$ to $\mathbb{Z}^{+} \cup \{0\}$.
It is defined so that
\begin{displaymath}
\la_{M}(X)= \ra_{M}(X) + \ra_{M}(E - X) - \ra(M)
\end{displaymath}
for any subset $X \subseteq E$.
Note that $\la(X) = \la(E - X)$ and
$\la_{M^{*}}(X) = \la_{M}(X)$ for any subset $X \subseteq E$.
It is well known, and easy to verify, that the connectivity
function of $M$ is submodular.
That is, for all $X,\, Y \subseteq E$, the inequality
\begin{displaymath}
\la(X \cap Y) + \la(X \cup Y) \leq \la(X) + \la(Y)
\end{displaymath}
is satisfied.

We say that a subset $X \subseteq E$ is
\emph{$k$\dash separating} or a \emph{$k$\dash separator} of $M$ if
$\la(X) < k$, and we say that a partition
$(X,\, E - X)$ is a \emph{$k$\dash separation} of $M$ if $X$ is
$k$\dash separating and $|X|,\, |E - X| \geq k$.
A $k$\dash separator $X$ or a $k$\dash separation
$(X,\, E - X)$ is \emph{exact} if $\la(X) = k-1$.
A matroid $M$ is \emph{$n$\dash connected} if $M$ has no
$k$\dash separation for any $k < n$.
We define a \emph{$k$\dash partition} of $M$ to be a
partition $(X_{1},\, X_{2},\dots, X_{n})$ of $E$ such that
$X_{i}$ is $k$\dash separating for all $1 \leq i \leq n$.
We say that the $k$\dash partition $(X_{1},\, X_{2},\dots, X_{n})$
is \emph{exact} if each $k$\dash separator $X_{i}$ is exact.

The next result is easy.

\begin{prop}
\label{prop1}
Let $N$ be a minor of the matroid $M$ and let $X$ be a
subset of $E(M)$.
Then $\la_{N}(E(N) \cap X) \leq \la_{M}(X)$.
\end{prop}

\begin{prop}
\label{guts}
Suppose that $M$ is a matroid and that $(X,\, Y,\, z)$ is a partition
of $E(M)$.
If $\la(X) = \la(Y)$ then $z$ is in $\cl(X) \cap \cl(Y)$ or in
$\cl^{*}(X) \cap \cl^{*}(Y)$, but not both.
\end{prop}

\begin{proof}
Since
\begin{displaymath}
\la(X) = \ra(X) +\ra(Y \cup z) - \ra(M)
= \ra(X \cup z) + \ra(Y) - \ra(M) = \la(Y)
\end{displaymath}
it follows that $\ra(Y \cup z) - \ra(Y) = \ra(X \cup z) - \ra(X)$.
Therefore, $z \in \cl(X)$ if and only if $z \in \cl(Y)$.
In the case that $z \notin \cl(X)$ and $z \notin \cl(Y)$ then
\begin{multline*}
\ra^{*}(Y \cup z) - \ra^{*}(Y) = (|Y \cup z| + \ra(X) - \ra(M))\\
- (|Y| + \ra(X \cup z) - \ra(M)) = 1 + \ra(X) - \ra(X \cup z) = 0.
\end{multline*}
Thus $z \in \cl^{*}(Y)$.
The same argument shows that $z \in \cl^{*}(X)$.

Finally we note that $z \in \cl^{*}(X)$ if and only if
$z \notin \cl(Y)$.
Thus $\cl(X) \cap \cl(Y)$ and $\cl^{*}(X) \cap \cl^{*}(Y)$
are disjoint.
\end{proof}

The next result is well known, and follows without
difficulty from the dual of~\cite[Lemma~2.5]{flowers}.

\begin{prop}
\label{cosegment}
Suppose that $X$ is an exactly $3$\dash separating set of
the $3$\dash connected matroid $M$.
Suppose also that $A \subseteq E(M) - X$.
If $|A| \geq 3$ and $A \subseteq \cl^{*}(X)$ then $A$ is a
cosegment of $M$.
\end{prop}

\begin{dfn}
\label{dfn1}
Suppose that $M$ is a matroid and that $x \in E(M)$.
Let $(X_{1},\, X_{2})$ be a partition of $E(M) - x$ such that
there is a positive integer $k$ with the property that:
\begin{enumerate}[(i)]
\item $\la(X_{1}) = \la(X_{2}) = k - 1$;
\item $\ra(X_{1}),\, \ra(X_{2}) \geq k$; and,
\item $x \in \cl(X_{1}) \cap \cl(X_{2})$.
\end{enumerate}
In this case $(X_{1},\, X_{2},\, x)$ is a
\emph{vertical $k$\dash partition} of $M$.
\end{dfn}

The next result is well known and easy to prove.

\begin{prop}
\label{contr2}
Let $M$ be a $3$\dash connected matroid and suppose that
$\si(M / x)$ is not $3$\dash connected for some $x \in E(M)$.
Then there exists a vertical $3$\dash partition
$(X_{1},\, X_{2},\, x)$ of $M$.
\end{prop}

\begin{prop}
\label{vertcl}
Suppose that $(X_{1},\, X_{2},\, x)$ is vertical $k$\dash partition
of the $k$\dash connected matroid $M$.
Let $A$ be a subset of $\cl(X_{2} \cup x)$.
Then $(X_{1} - A,\, (X_{2} \cup A) - x,\, x)$ is also
a vertical $k$\dash partition of $M$.
\end{prop}

\begin{proof}
Suppose that $z$ is some element in $X_{1} \cap A$.
Then $\la(X_{1} - z)$ is either $k - 2$ or $k - 1$.
If $\la(X_{1} - z) = k - 2$ then $(X_{1} - z,\, X_{2} \cup \{x,\,z\})$
is a $(k - 1)$\dash separation of $M$, a contradiction.
Hence $\la(X_{1} - z) = k - 1$ which implies that
$\ra(X_{1} - z) = \ra(X_{1})$.
Thus $\cl(X_{1} - z) = \cl(X_{1})$, and hence $x \in \cl(X_{1} - z)$.
It follows that $(X_{1} - z,\, X_{2} \cup z,\, x)$ is a vertical
$k$\dash partition of $M$.
By continuing to transfer elements in $X_{1} \cap A$
from $X_{1}$ into $X_{2}$ we eventually conclude that
$(X_{2} - A,\, (X_{2} \cup A) - x,\, x)$ is a
vertical $k$\dash partition of $M$, as desired.
\end{proof}

Suppose that $M_{1}$ and $M_{2}$ are matroids such that
$E(M_{1}) \cap E(M_{2}) = \{p\}$.
Then we can define the \emph{parallel connection} of $M_{1}$ and $M_{2}$,
denoted by $P(M_{1},\, M_{2})$.
The ground set of $P(M_{1},\, M_{2})$ is $E(M_{1}) \cup E(M_{2})$.
If $p$ is a loop in neither $M_{1}$ nor $M_{2}$ then
the circuits of $P(M_{1},\, M_{2})$ are exactly the circuits
of $M_{1}$, the circuits of $M_{2}$, and sets of the form
$(C_{1} - p) \cup (C_{2} - p)$, where $C_{i}$ is a circuit of
$M_{i}$ such that $p \in C_{i}$ for $i = 1,\, 2$.
If $p$ is a loop in $M_{1}$ then $P(M_{1},\, M_{2})$
is defined to be the direct sum of $M_{1}$ and $M_{2} / p$.
Similarly, if $p$ is a loop in $M_{2}$ then $P(M_{1},\, M_{2})$
is defined to be the direct sum of $M_{1} / p$ and $M_{2}$.
We say that $p$ is the \emph{basepoint} of the parallel connection.
It is clear that $P(M_{1},\, M_{2}) = P(M_{2},\, M_{1})$.

The next result follows from~\cite[Proposition~7.1.15~(v)]{oxley}.

\begin{prop}
\label{prop6}
Suppose that $M_{1}$ and $M_{2}$ are matroids such that
$E(M_{1}) \cap E(M_{2}) = \{p\}$.
If $e \in E(M_{1}) - p$ then
$P(M_{1},\, M_{2}) \del e = P(M_{1} \del e,\, M_{2})$ and
$P(M_{1},\, M_{2}) / e = P(M_{1} / e,\, M_{2})$.
\end{prop}

Assume that $M_{1}$ and $M_{2}$ are matroids such that
$E(M_{1}) \cap E(M_{2}) = \{p\}$.
If $p$ is not a loop or a coloop in either $M_{1}$ or $M_{2}$
then $P(M_{1},\, M_{2}) \del p$ is the
\emph{$2$\dash sum} of $M_{1}$ and $M_{2}$, denoted by
$M_{1} \oplus_{2} M_{2}$.
We say that $p$ is the \emph{basepoint} of the
$2$\dash sum.

The next result follows from~\cite[(2.6)]{seymour3}.

\begin{prop}
\label{prop2}
If $(X_{1},\, X_{2})$ is an exact $2$\dash separation of a matroid
$M$ then there exist matroids $M_{1}$ and $M_{2}$ on the
ground sets $X_{1} \cup p$ and $X_{2} \cup p$ respectively, where
$p$ is in neither $X_{1}$ nor $X_{2}$, such that $M$ is equal to
$M_{1} \oplus_{2} M_{2}$.
\end{prop}

\begin{prop}
\label{prop5}
Suppose that $N$ is a $3$\dash connected matroid.
Let $M$ be a matroid with a vertical $3$\dash partition
$(X_{1},\, X_{2},\, x)$ such that $N$ is a minor of $M / x$.
Then either $|E(N) \cap X_{1}| \leq 1$, or $|E(N) \cap X_{2}| \leq 1$.
\end{prop}

\begin{proof}
Since $(X_{1},\, X_{2})$ is a $2$\dash separation of $M / x$ the result
follows immediately from Proposition~\ref{prop1}.
\end{proof}

\begin{lemma}
\label{smallside}
Suppose that $N$ is a $3$\dash connected matroid such that
$|E(N)| \geq 2$.
Let $M$ be a matroid with a vertical $3$\dash partition
$(X_{1},\, X_{2},\, x)$ such that $N$ is a minor of $M / x$.
If $|E(N) \cap X_{1}| \leq 1$ then $M / x / e$ has an $N$\dash minor
for every element $e \in X_{1} - \cl_{M}(X_{2})$.
\end{lemma}

\begin{proof}
Since $(X_{1},\, X_{2})$ is an exact $2$\dash separation of $M / x$,
it follows from Proposition~\ref{prop2} that $M / x$ is the
$2$\dash sum of matroids $M_{1}$ and $M_{2}$ along the basepoint $p$,
where $E(M_{1}) = X_{1} \cup p$ and $E(M_{2}) = X_{2} \cup p$.
Thus $M / x = P(M_{1},\, M_{2}) \del p$.

Suppose that $E(N) \cap X_{1} = \emptyset$.
Then there is a partition $(A,\, B)$ of $X_{1}$ such that
$N$ is a minor of $M / x / A \del B$.
Suppose that $p$ is a loop in $M_{1} / A \del B$.
Proposition~\ref{prop6} implies that
\begin{displaymath}
M / x / A \del B = P(M_{1} / A \del B,\, M_{2}) \del p.
\end{displaymath}
Now the definition of parallel connection implies that
$M / x / A \del B$ is isomorphic to $M_{2} / p$.
It is easily seen that if $e \in X_{1}$ then there is a minor $M'$
of $M_{1} / e$ such that $E(M') = \{p\}$ and $p$ is a loop of
$M'$.
Proposition~\ref{prop6} implies that $P(M',\, M_{2}) \del p$ is
a minor of $M / x / e$.
But $P(M',\, M_{2}) \del p$ is isomorphic to $M_{2} / p$, so
$M / x / e$ has an $N$\dash minor.

Next we suppose that $p$ is a coloop of $M_{1} / A \del B$.
Then, by definition of the parallel connection, $M / x / A \del B$
is isomorphic to $M_{2} \del p$.
Suppose that $e \in X_{1} - \cl(X_{2})$.
Since $p$ is not a coloop of $M_{2}$ it follows easily that
$p \in \cl_{M}(X_{2})$.
Thus $e$ is not parallel to $p$ in $M_{1}$.
Therefore there is a minor $M'$ of $M_{1} / e$
such that $E(M') = \{p\}$ and $p$ is a coloop of
$M'$.
Again using Proposition~\ref{prop6} we see that $P(M',\, M_{2}) \del p$
is a minor of $M / x / e$.
But since $P(M',\, M_{2}) \del p$ is isomorphic to $M_{2} \del p$ we
deduce that $M / x / e$ has an $N$\dash minor.

Now we assume that $|E(N) \cap X_{1}| = 1$ and that $z$ is the unique
element in $E(N) \cap X_{1}$.
There is a partition $(A,\, B)$ of $X_{1} - z$ such that
$N$ is a minor of $M / x / A \del B$.
It follows from Proposition~\ref{prop6} that
$P(M_{1} / A \del B,\, M_{2}) \del p$ has an $N$\dash minor.
Consider the matroid $M_{1} / A \del B$.
If $\{z,\, p\}$ is not a parallel pair in this matroid then
$z$ must be a loop or coloop in $P(M_{1} / A \del B,\, M_{2}) \del p$.
This implies that $z$ is a loop or coloop in $N$, a contradiction as
$N$ is $3$\dash connected and $|E(N)| \geq 2$.
Therefore $z$ and $p$ are parallel in
$M_{1} / A \del B$, and therefore
$P(M_{1} / A \del B,\, M_{2}) \del p$ is isomorphic to
$M_{2}$.
Thus $M_{2}$ has an $N$\dash minor.

Since $p$ is not a loop or coloop of $M_{1}$
there is a circuit of size
at least two in $M_{1}$ that contains $p$.
Suppose that $e \in X_{1} - \cl_{M}(X_{2})$.
Then $e$ cannot be parallel to $p$ in $M_{1}$, so
$M_{1} / e$ has a circuit of size at least two
that contains $p$.
Hence there is a minor $M'$ of $M_{1} / e$ such that $p \in E(M')$
and $M'$ consists of a parallel pair.
Proposition~\ref{prop6} implies that $P(M',\, M_{2}) \del p$
is a minor of $M / x / e$.
But $P(M',\, M_{2}) \del p$ is isomorphic to $M_{2}$, so
$M / x / e$ has an $N$\dash minor.
\end{proof}

\begin{dfn}
\label{dfn2}
Suppose that $M$ is a matroid and that $A$ and $B$ are
subsets of $E(M)$.
The \emph{local connectivity} between $A$ and $B$,
denoted by $\sqcap(A,\, B)$, is defined to be
$\ra(A) + \ra(B) - \ra(A\cup B)$.
Equivalently, $\sqcap(A,\, B)$ is equal to
$\la_{M|(A \cup B)}(A)$.
\end{dfn}

\begin{prop}
\textup{\cite[Lemma~2.4~(iv)]{flowers}}
\label{flowers3}
Let $M$ be a matroid and let $(A,\, B,\, C)$ be a partition
of $E(M)$.
Then $\sqcap(A,\, B) + \la(C) = \sqcap(A,\, C) + \la(B)$.
Hence $\sqcap(A,\, B) = \sqcap(A,\, C)$ if and only if
$\la(B) = \la(C)$.
\end{prop}

\begin{corollary}
\label{flowers1}
Let $(X,\, Y,\, Z)$ be an exact $3$\dash partition of the
$3$\dash connected matroid $M$.
Then $\sqcap(X,\, Y) = \sqcap(X,\, Z) = \sqcap(Y,\, Z)$.
\end{corollary}

\begin{prop}
\label{prop7}
Suppose that $M$ is a matroid and that $X$ and $Y$ are disjoint
subsets of $E(M)$ such that $\sqcap(X,\, Y) = 1$.
If $x,\, y \in X \cap \cl(Y)$ then
$\ra(\{x,\, y\}) \leq 1$.
\end{prop}

\begin{proof}
Assume that $\ra(\{x,\, y\}) = 2$.
Let $X' = \cl(X)$ and $Y' = \cl(Y)$.
It is easy to see that $\ra(X' \cup Y') = \ra(X \cup Y)$.
However
\begin{displaymath}
\ra(X' \cup  Y') \leq \ra(X') + \ra(Y') - \ra(X' \cap Y')
\leq \ra(X) + \ra(Y) - 2 = \ra(X \cup Y) - 1.
\end{displaymath}
This contradiction completes the proof.
\end{proof}

We conclude this section by stating a fundamental tool in the
study of $3$\dash connected matroids, due to Bixby~\cite{Bixby}.

\begin{theorem}[Bixby's Lemma]
\label{bixby}
Let $M$ be a $3$\dash connected matroid and suppose that
$x$ is an element of $E(M)$.
Then either $\si(M / x)$ or $\co(M \del x)$ is $3$\dash connected.
\end{theorem}

\section{Segment-cosegment pairs}
\label{segments}

Suppose that $M$ is a matroid.
Recall that $L$ is a segment of $M$ if $|L| \geq 3$ and
every three-element subset of $L$ is a circuit of $M$, and that
$L^{*}$ is a cosegment of $M$ if $|L^{*}| \geq 3$ and every
three-element subset of $L^{*}$ is a cocircuit.
We restate the definition of segment-cosegment pairs
given in Section~1.

\begin{dfn}
\label{dfn3}
Suppose that $L = \{x_{1},\ldots, x_{t}\}$ is a segment of the matroid
$M$ and there is a set $L^{*} = \{y_{1},\ldots, y_{t}\}$ with the
property that $L \cap L^{*} = \emptyset$ and
$(\cl(L) - x_{i}) \cup y_{i}$ is a cocircuit of $M$ for all
$i \in \{1,\ldots, t\}$.
In this case we say that $(L,\, L^{*})$ is a
\emph{segment-cosegment pair} of $M$.
\end{dfn}

In a $3$\dash connected matroid a segment-cosegment pair is an
example of a `crocodile', a structure that provides a collection of
equivalent $3$\dash separations.
`Crocodiles' were considered by Hall, Oxley, and
Semple~\cite{sequences}.
The next result explains the name segment-cosegment pair.

\begin{prop}
\label{prop9}
Suppose that $(L,\, L^{*})$ is a segment-cosegment pair of the
$3$\dash connected matroid $M$.
Then $L^{*}$ is a cosegment of $M$.
\end{prop}

\begin{proof}
Suppose that $y_{i} \in L^{*}$.
The definition of a segment-cosegment pair means that
$y_{i} \in \cl^{*}(\cl(L))$.
Thus $L^{*} \subseteq \cl^{*}(\cl(L))$.
Moreover $\cl(L)$ is exactly $3$\dash separating in $M$.
The result follows by Proposition~\ref{cosegment}.
\end{proof}

\begin{prop}
\label{prop10}
Suppose that $(L,\, L^{*})$ is a segment-cosegment pair of the
$3$\dash connected matroid $M$.
Then $M / \cl(L)$ is $3$\dash connected.
\end{prop}

\begin{proof}
Suppose that $L = \{x_{1},\ldots, x_{t}\}$ and
$L^{*} = \{y_{1},\ldots, y_{t}\}$.
Assume that $M / \cl(L)$ is not $3$\dash connected, so that
$(X_{1},\, X_{2})$ is a $k$\dash separation of $M / \cl(L)$
for some $k \leq 2$.
Let $L_{0} = \cl(L)$.
Note that for $i \in \{1,\, 2\}$ we have
\begin{displaymath}
\ra_{M / L_{0}}(X_{i}) = \ra_{M}(X_{i} \cup L_{0}) - \ra_{M}(L_{0})
= \ra_{M}(X_{i}) - \sqcap_{M}(X_{i},\, L_{0}),
\end{displaymath}
so
$\ra_{M}(X_{i}) = \ra_{M / L_{0}}(X_{i}) + \sqcap_{M}(X_{i},\, L_{0})$.

Suppose that $\sqcap_{M}(X_{1},\, L_{0}) = 0$.
Then $\ra_{M}(X_{1}) = \ra_{M / L_{0}}(X_{1})$ and
$\ra_{M}(X_{2} \cup L_{0}) = \ra_{M / L_{0}}(X_{2}) + 2$,
so
\begin{multline*}
\la_{M}(X_{1}) = \ra_{M / L_{0}}(X_{1})
+ (\ra_{M / L_{0}}(X_{2}) + 2) - (\ra(M / L_{0}) + 2)\\
= \la_{M / L_{0}}(X_{1}) < k.
\end{multline*}
This is a contradiction as $M$ is $3$\dash connected.
By using a symmetric argument we can conclude that
$\sqcap_{M}(X_{i},\, L_{0}) > 0$ for all
$i \in \{1,\,2\}$.

Suppose that $x_{i} \in \cl_{M}(X_{1})$ for some
$i \in \{1,\ldots, t\}$.
Then there is a circuit $C_{1} \subseteq X_{1} \cup x_{i}$
such that $x_{i} \in C_{1}$.
For all $k \in \{1,\ldots, t\} - i$ the set
$(L_{0} - x_{k}) \cup y_{k}$ is a cocircuit.
It cannot be the case that $C_{1}$ meets this cocircuit
in a single element, so $y_{k} \in X_{1}$ for all
$k \in \{1,\ldots, t\} - i$.

Now suppose that $x_{j} \in \cl_{M}(X_{2})$ for some
$j \in \{1,\ldots, t\}$.
By using the same arguments as above
we can conclude that
$L^{*} - y_{j} \subseteq X_{2}$.
As $L^{*} - y_{i}$ and $L^{*} - y_{j}$ have a non-empty
intersection this is a contradiction.
Therefore $\cl_{M}(X_{2}) \cap L = \emptyset$.
Note that $\sqcap(X_{2},\, L_{0}) \leq 2$ because $\ra(L_{0}) = 2$.
If $\sqcap(X_{2},\, L_{0})$ were two, it would follow that
$L_{0} \subseteq \cl(X_{2})$.
Hence $\sqcap(X_{2},\, L_{0}) = 1$.

Let $j$ be an element of $\{1,\ldots, t\} - i$.
Then $L_{0} \subseteq \cl_{M}(X_{2} \cup x_{j})$, and there
must be a circuit $C_{2} \subseteq X_{2} \cup \{x_{i},\,x_{j}\}$
such that $\{x_{i},\, x_{j}\} \subseteq C_{2}$.
But then $C_{2}$ meets the cocircuit
$(L_{0} - x_{j}) \cup y_{j}$ in a single element, $x_{i}$.
From this contradiction we conclude that
$\cl_{M}(X_{1}) \cap L = \emptyset$, and by symmetry
$\cl_{M}(X_{2}) \cap L = \emptyset$.
This means that
\begin{displaymath}
\sqcap_{M}(X_{1},\, L_{0}) = \sqcap_{M}(X_{2},\, L_{0}) = 1. 
\end{displaymath}

It must be the case that $x_{2} \in \cl_{M}(X_{1} \cup x_{1})$,
and there is a circuit $C_{3} \subseteq X_{1} \cup \{x_{1},\, x_{2}\}$
such that $\{x_{1},\, x_{2}\} \subseteq C_{3}$.
Since $(L_{0} - x_{1}) \cup y_{1}$ is a cocircuit we conclude that
$y_{1} \in X_{1}$.
But we can use an identical argument to show that
$y_{1} \in X_{2}$.
This contradiction completes the proof.
\end{proof}

We now restate the definition of a spore.

\begin{dfn}
\label{dfn4}
Suppose that $P$ is a rank-one flat of a matroid $M$
and that $s$ is an element of $E(M)$ such that $P \cup s$ is a
cocircuit.
Then we say that $(P,\, s)$ is a \emph{spore}.
\end{dfn}

Recall from Section~\ref{intro} that a matroid $M$ is
$3$\dash connected up to a unique spore if it contains
a single spore $(P,\, s)$, and whenever $(X,\, Y)$
is a $k$\dash separation of $M$ for some $k < 3$ then
either $X \subseteq P \cup s$ or $Y \subseteq P \cup s$.

\begin{lemma}
\label{crocspore}
Suppose that $(L,\, L^{*})$ is a segment-cosegment pair
of the $3$\dash connected matroid $M$ where $|E(M) - \cl(L)| \geq 4$.
Let $L = \{x_{1},\ldots, x_{t}\}$ and
$L^{*} = \{y_{1},\ldots, y_{t}\}$.
Then $M / x_{i}$ is $3$\dash connected up to a unique
spore $(\cl(L) - x_{i},\, y_{i})$, for all $i \in \{1,\ldots, t\}$.
\end{lemma}

\begin{proof}
Let $E$ be the ground set of $M$ and let $L_{0} = \cl(L)$.
We will show that $M / x_{i}$ is $3$\dash connected up to the
unique spore $(L_{0} - x_{i},\, y_{i})$.
Certainly $(L_{0} - x_{i},\, y_{i})$ is a spore of $M / x_{i}$.
Suppose that $(P,\, s)$ is a spore of $M / x_{i}$ that is distinct
from $(L_{0} - x_{i},\, y_{i})$.

We initially assume that $L_{0} - x_{i} = P$.
Thus $s \ne y_{i}$.
As $(L_{0} - x_{i}) \cup s$ and $(L_{0} - x_{i}) \cup y_{i}$ are both
cocircuits of $M / x_{i}$ it follows that $E - (L_{0} \cup \{s,\, y_{i}\})$
is the intersection of two hyperplanes of $M / x_{i}$.
Thus
\begin{displaymath}
\ra_{M / x_{i}}(E - (L_{0} \cup \{s,\, y_{i}\})) \leq \ra(M / x_{i}) - 2.
\end{displaymath}
and therefore
\begin{displaymath}
\ra_{M / L_{0}}(E - (L_{0} \cup \{s,\, y_{i}\})) \leq \ra(M / x_{i}) - 2 =
\ra(M / L_{0}) - 1.
\end{displaymath}
Hence $\{s,\, y_{i}\}$ contains a cocircuit in $M / L_{0}$.
Therefore $M / L_{0}$ contains a cocircuit of size at most two, a
contradiction as $M / L_{0}$ is $3$\dash connected
by Proposition~\ref{prop10}, and $|E(M / L_{0})| \geq 4$.

Now we must assume that $L_{0} - x_{i} \ne P$.
Hence $P \cup x_{i}$ is a rank-two flat of $M$ that meets $L_{0}$
in exactly one element, $x_{i}$.
Suppose that $P$ contains a single element $p$.
Then $\{p,\, s\}$ is a cocircuit of $M$, a contradiction.
Therefore $P \cup x_{i}$ contains at least one triangle.
Suppose that $P$ does not contain $y_{j}$, where $j \ne i$.
Then there is a triangle in $P \cup x_{i}$ that meets the
cocircuit $(L_{0} - x_{j}) \cup y_{j}$ in exactly one element, $x_{i}$.
This contradiction shows that $L^{*} - y_{i} \subseteq P$.

Assume that $t > 3$.
As $L^{*}$ is a cosegment there is a triad of $M$ contained in
$L^{*} - y_{i}$.
However this triad is also contained in the segment $P \cup x_{i}$,
and is therefore a triangle.
But $|E(M)| > 4$ and a $3$\dash connected matroid with at least
five elements cannot contain a triangle that is also a triad.
This contradiction shows that $t = 3$.

Suppose $j \in \{1,\, 2,\, 3\}$ and that $j \ne i$.
If $|P| > 2$ then there is a triangle contained in $P$ that
contains $y_{j}$.
However this triangle would meet the cocircuit $(L_{0} - x_{j}) \cup y_{j}$
in exactly one element.
Thus $|P| = 2$, and $P = L^{*} - y_{i}$.

Suppose that $j,\, k \in \{1,\, 2,\, 3\}$ and neither
$j$ nor $k$ is equal to $i$.
Then $L_{0} \cup P$ contains the two cocircuits
$(L_{0} - x_{j}) \cup y_{j}$ and $(L_{0} - x_{k}) \cup y_{k}$.
Hence $\ra_{M}(E - (L_{0} \cup P)) \leq \ra(M) - 2$.
However it is easy to see that $\ra_{M}(L_{0} \cup P) = 3$.
As $|P| = 2$ it follows that $E - (L_{0} \cup P)$ contains
at least two elements.
Thus $(L_{0} \cup P,\, E - (L_{0} \cup P))$ is a
$2$\dash separation of $M$, a contradiction.

We have shown that $(L_{0} - x_{i},\, y_{i})$ is the unique
spore of $M / x_{i}$.
Next we show that $M / x_{i}$ is $3$\dash connected up to this spore.
Suppose that $(X,\, Y)$ is a $k$\dash separation of $M / x_{i}$
for some $k < 3$.
By relabeling if necessary we will assume that $y_{i} \in X$.
Assume that the result is false, so that neither $X$ nor
$Y$ is contained in $(L_{0} - x_{i}) \cup y_{i}$.
Therefore $X$ contains at least one element from
$E - (L_{0} \cup y_{i})$.
As $M / L_{0}$ is $3$\dash connected by Proposition~\ref{prop10}
we deduce from Proposition~\ref{prop1} that either $X - L_{0}$
or $Y - L_{0}$ contains at most one element.
We have already concluded that $X - L_{0}$ contains at least
two elements (as $y_{i} \in X$), so $Y - L_{0}$ contains
precisely one element.
As $M$ is $3$\dash connected it contains no parallel pairs,
so $M / x_{i}$ contains no loops.
Therefore $\ra_{M / x_{i}}(Y) = 2$, and hence
$\ra_{M / x_{i}}(X) \leq \ra(M / x_{i}) - 1$.
Thus $Y$ contains a cocircuit of $M / x_{i}$.
As $M / x_{i}$ has no coloops, and any cocircuit that meets a
parallel class contains that parallel class it follows that
$L_{0} - x_{i} \subseteq Y$.
Let $s$ be the single element in $Y - L_{0}$.
It cannot be the case that $Y$ is a cocircuit in $M / x_{i}$, for
that would imply that $(L_{0} - x_{i},\, s)$ is a spore of $M / x_{i}$
that differs from $(L_{0} - x_{i},\, y_{i})$, contradicting
our earlier conclusion.
Now we see that $Y - s = L_{0} - x_{i}$ must be a cocircuit of
$M / x_{i}$, but this is a contradiction as
$L_{0} - x_{i}$ is properly contained in the cocircuit
$(L_{0} - x_{i}) \cup y_{i}$.
The completes the proof.
\end{proof}

The next result shows that Theorem~\ref{thm1} is a
consequence of Theorem~\ref{main}.

\begin{prop}
\label{prop4}
Suppose that $(L,\, L^{*})$ is a segment-cosegment pair of a
matroid $M$, and that $M / \cl(L)$ is $3$\dash connected
and $|E(M) - \cl(L)| \geq 4$.
Let $L = \{x_{1},\ldots, x_{t}\}$ and
$L^{*} = \{y_{1},\ldots, y_{t}\}$.
Then $\co(\si(M / x_{i})) \iso M / \cl(L)$ for any element
$x_{i} \in L$.
\end{prop}

\begin{proof}
Let $L_{0} = \cl(L)$ and let $x_{j} \ne x_{i}$ be an element of $L$.
Suppose that $P$ and $S$ are disjoint subsets of $E(M) - x_{i}$
chosen so that $\co(\si(M / x_{i})) \iso M / x_{i} \del P / S$.
As $L_{0} - x_{i}$ is a parallel class in $M / x_{i}$ we may assume
that $L_{0} - \{x_{i},\, x_{j}\} \subseteq P$ and that $x_{j} \notin P$.
We may assume that $y_{i} \notin P$, and hence $\{x_{j},\, y_{i}\}$ is
a union of cocircuits in $M / x_{i} \del P$.
Therefore we may assume $x_{j} \in S$.
Since the elements in $L_{0} - \{x_{i},\, x_{j}\}$ are loops
in $M / x_{i} / x_{j}$ it follows that
\begin{displaymath}
M / x_{i} \del P / S = M / x_{i} / x_{j} / (L_{0} - \{x_{i},\, x_{j}\})
\del (P - (L_{0} - \{x_{i},\, x_{j}\})) / (S - x_{j}).
\end{displaymath}
This last matroid is equal to
$M / L_{0} \del (P - (L_{0} - \{x_{i},\, x_{j}\})) / (S - x_{j})$.
Since $M / L_{0}$ is $3$\dash connected and the elements
in $P - (L_{0} - \{x_{i},\, x_{j}\})$ are either loops or parallel
elements in $M / L_{0}$ it follows that
$P = L_{0} - \{x_{i},\, x_{j}\}$.
Thus $M / x_{i} \del P / S = M / L_{0} / (S - x_{j})$.
But $M / L_{0}$ is $3$\dash connected, so $S - x_{j}$ must be
empty.
Thus $M / L_{0} \iso \co(\si(M / x_{i}))$, as desired.
\end{proof}

\section{Preliminary lemmas}
\label{prelim}

\begin{prop}
\label{biglem}
Suppose that $C^{*}$ is a cocircuit of the $3$\dash connected
matroid $M$.
Assume that $(X_{1},\, X_{2},\, x)$ is a vertical $3$\dash partition of
$M$ such that $x\in C^{*}$.
Then $C^{*} \cap (X_{1} - \cl(X_{2})) \ne \emptyset$ and
$C^{*} \cap (X_{2} - \cl(X_{1})) \ne \emptyset$.
\end{prop}

\begin{proof}
Note that $\ra(X_{1}),\, \ra(X_{2}) \geq 3$ implies that
$|E(M)| \geq 4$, so every circuit and cocircuit of $M$ contains
at least three elements.
Let $X$ be $X_{1} - \cl(X_{2})$.
The fact that $\ra(X_{1}) \geq 3$ implies that $X$ contains a
cocircuit, so $|X| \geq 3$.
Suppose that $x$ is not in $\cl(X)$.
Then $\ra(X) < \ra(X_{1})$.
Since $|X| \geq 3$ this implies that $(X,\, \cl(X_{2}))$ is
a $2$\dash separation of $M$, a contradiction.

Now suppose that $C^{*} \subseteq \cl(X_{2})$.
Then as $x \in \cl(X)$ and $x \in C^{*}$ there is a
circuit in $M$ that meets $C^{*}$ in exactly one element, $x$.
This is a contradiction.
The same argument shows that
$C^{*} \cap (X_{2} - \cl(X_{1})) \ne \emptyset$, so
the proposition holds.
\end{proof}

\begin{dfn}
\label{dfn5}
Suppose that $M$ is a $3$\dash connected matroid and that $A$
is a subset of $E(M)$.
A \emph{minimal partition} with respect to $A$ is a
vertical $3$\dash partition $(X_{1},\, X_{2},\, x)$ of $M$
that satisfies the following properties:
\begin{enumerate}[(i)]
\item $x \in A$;
\item if $(Y_{1},\, Y_{2},\, y)$ is a vertical $3$\dash partition
of $M$ such that $y \in A \cap (X_{1} \cup x)$ and
$X_{2} \cap Y_{1} = \emptyset$, then
$(Y_{1},\, Y_{2},\, y) = (X_{1},\, X_{2},\, x)$; and,
\item if $(Y_{1},\, Y_{2},\, y)$ is a vertical $3$\dash partition
of $M$ such that $y \in A \cap (X_{1} \cup x)$ and
$X_{2} \cap Y_{2} = \emptyset$ then
$(Y_{2},\, Y_{1},\, y) = (X_{1},\, X_{2},\, x)$.
\end{enumerate}
\end{dfn}

If there is no ambiguity we will refer to a minimal partition
with respect to $A$ as a \emph{minimal partition}.

\begin{lemma}
\label{minimal}
Suppose that $M$ is a $3$\dash connected matroid and that $A$ is a
subset of $E(M)$.
Suppose that for some element $z \in A$ there is a
vertical $3$\dash partition $(Z_{1},\, Z_{2},\, z)$ of $M$.
Let $Z = Z_{1} - \cl(Z_{2})$.
Then there is a minimal partition $(X_{1},\, X_{2},\, x)$ with respect
to $A$ such that $X_{1} \subseteq Z$ and $x \in A \cap (Z \cup z)$.
\end{lemma}

\begin{proof}
Let \mcal{Z} be the family of vertical $3$\dash partitions
$(S_{1},\, S_{2},\, z)$ with the property that
$S_{1} \subseteq Z_{1}$.
Choose $(Z_{1}',\, Z_{2}',\, z)$ from \mcal{Z} so that if
$(S_{1},\, S_{2},\, z)$ is in \mcal{Z}, then
$S_{1}$ is not properly contained in $Z_{1}'$.
Observe that Proposition~\ref{vertcl} implies that
$Z_{1}' \subseteq Z$.

Let $\mcal{S}$ be the family of vertical $3$\dash partitions
$(S_{1},\, S_{2},\, s)$ with $s \in A \cap (Z_{1}' \cup z)$.
Let $\mcal{S}_{0}$ be the set of vertical $3$\dash partitions
$(S_{1},\, S_{2},\, s)$ in \mcal{S} with the property that
either $S_{1} \subseteq Z_{1}'$ or $S_{2} \subseteq Z_{1}'$.
Without loss of generality we will assume that
if $(S_{1},\, S_{2},\, s)$ is in $\mcal{S}_{0}$ then
$S_{1} \subseteq Z_{1}'$.
Suppose that $(S_{1},\, S_{2},\, z)$ is a member of
$\mcal{S}_{0}$.
Then our choice of $(Z_{1}',\, Z_{2}',\, z)$ means that
$S_{1} = Z_{1}'$ and $S_{2} = Z_{2}'$.
If $(Z_{1}',\, Z_{2}',\, z)$ is the only member of
$\mcal{S}_{0}$ then we can set $(X_{1},\, X_{2},\, x)$
to be $(Z_{1}',\, Z_{2}',\, z)$, and we will be done.
Therefore we will assume that there is at least one
vertical $3$\dash partition $(S_{1},\, S_{2},\, s)$ in
$\mcal{S}_{0}$ such that $s \ne z$.
Let $\mcal{S}_{1}$ be the collection of such partitions.

We now let $(X_{1},\, X_{2},\, x)$ be a vertical $3$\dash partition
in $\mcal{S}_{1}$ chosen so that if
$(S_{1},\, S_{2},\, s) \in \mcal{S}_{1}$, then
$S_{1} \cup s$ is not properly contained in $X_{1} \cup x$.
We will prove that $(X_{1},\, X_{2},\, x)$ is the desired
vertical $3$\dash partition.

It is certainly true that $X_{1} \subseteq Z$.
If there is some element $e$ in $X_{1} \cap \cl(X_{2} \cup x)$
then $(X_{1} - e,\, X_{2} \cup e,\, x)$ is a vertical $3$\dash partition
by Proposition~\ref{vertcl}.
However this contradicts our choice of $(X_{1},\,X_{2},\,x)$.
Therefore $X_{2} \cup x$ is a flat.
We assume that $(Y_{1},\, Y_{2},\, y)$ is
a vertical $3$\dash partition and that $y \in A \cap (X_{1} \cup x)$.
As $X_{1} \subseteq Z_{1}'$ it follows that $y \in A \cap Z_{1}'$.
Our assumption on $(X_{1},\, X_{2},\, x)$ means that
neither $Y_{1} \cup y$ nor $Y_{2} \cup y$ can be properly
contained in $X_{1} \cup x$.

Suppose that $X_{2} \cap Y_{1} = \emptyset$.
Then $Y_{1} \cup y$ must be equal to $X_{1} \cup x$.
If $y \ne x$ then the fact that $y \in \cl(Y_{2})$ and
$Y_{2} = X_{2}$ means that $y \in \cl(X_{2})$,
which is a contradiction as $X_{2} \cup x$ is a flat.
Therefore $y = x$, so $(Y_{1},\, Y_{2},\, y)$ is equal to
$(X_{1},\, X_{2},\, x)$.
The same argument shows that if $X_{2} \cap Y_{2} = \emptyset$
then $(Y_{1},\, Y_{2},\, y) = (X_{2},\, X_{1},\, x)$.
Thus $(X_{1},\, X_{2},\, x)$ is the desired
minimal partition.
\end{proof}

\begin{prop}
\label{prop11}
Suppose that $M$ is a matroid and that
$A \subseteq E(M)$.
Suppose that $(X_{1},\, X_{2},\, x)$ is a minimal partition
with respect to $A$.
Then $X_{2} \cup x$ is a flat of $M$.
\end{prop}

\begin{proof}
Suppose that there is some element $z \in X_{1} \cap \cl(X_{2} \cup x)$.
Then $(X_{1} - z,\, X_{2} \cup z,\, x)$ is a vertical $3$\dash partition
of $M$ by Proposition~\ref{vertcl}.
This contradicts the fact that $(X_{1},\, X_{2},\, x)$ is a minimal
partition.
\end{proof}

\begin{lemma}
\label{crossetc}
Suppose that $M$ is a $3$\dash connected matroid and that
$A \subseteq E(M)$.
Suppose that $(X_{1},\, X_{2},\, x)$ is a minimal partition
with respect to $A$.
Suppose also that $(Y_{1},\, Y_{2},\, y)$ is a vertical
$3$\dash partition of $M$ such that $y \in A \cap X_{1}$ and
$x \in Y_{1}$.
Then the following statements hold:
\begin{enumerate}[(i)]
\item\label{cross}
$X_{i} \cap Y_{j} \ne \emptyset$ for all $i,j \in \{1,\, 2\}$;
\item\label{separating}
Each of $X_{1} \cap Y_{2}$, $(X_{1} \cap Y_{2}) \cup y$,
$X_{2} \cap Y_{1}$, $(X_{2} \cap Y_{1}) \cup x$, and
$X_{2} \cap Y_{2}$ is $3$\dash separating in $M$;
\item\label{separating2}
$(X_{1} \cap Y_{1}) \cup \{x,\, y\}$ is $4$\dash separating in $M$;
\item\label{air}
Neither $X_{1} \cap Y_{1}$ nor $X_{1} \cap Y_{2}$ is contained
in $\cl(X_{2})$,
$X_{1} \cap Y_{1} \nsubseteq \cl(Y_{2})$, and
$X_{1} \cap Y_{2} \nsubseteq \cl(Y_{1})$;
\item\label{rank} $\ra((X_{1} \cap Y_{2}) \cup y) = 2$; and,
\item\label{rank2} If $(X_{1} \cap Y_{1}) \cup \{x,\, y\}$ is
$3$\dash separating in $M$, then
$\ra((X_{1} \cap Y_{1}) \cup \{x,\, y\}) = 2$.
\end{enumerate}
\end{lemma}

\begin{proof}
We start by proving~\eqref{cross}.
Since $y \ne x$ the definition of a minimal partition means that
$X_{2} \cap Y_{1} \ne \emptyset$ and
$X_{2} \cap Y_{2} \ne \emptyset$.
Moreover $X_{2} \cup x$ is a flat of $M$ by Proposition~\ref{prop11},
and $y \in X_{1}$, so $y \notin \cl(X_{2} \cup x)$.
However $y \in \cl(Y_{1}) \cap \cl(Y_{2})$.
It follows that neither $Y_{1}$ nor $Y_{2}$ can be contained
in $X_{2} \cup x$.
Thus both $Y_{1}$ and $Y_{2}$ meet $X_{1}$.

Next we prove~\eqref{separating}.
Consider $X_{1} \cap Y_{2}$.
Since $\la(X_{1}) = 2$ and $\la(Y_{2}) = 2$ the
submodularity of the connectivity function implies that
$\la(X_{1} \cap Y_{2}) + \la(X_{1} \cup Y_{2}) \leq 4$.
If $X_{1} \cap Y_{2}$ is not $3$\dash separating then
$\la(X_{1} \cup Y_{2}) \leq 1$.
However $|X_{1} \cup Y_{2}| \geq 2$ and the complement of
$X_{1} \cup Y_{2}$ certainly contains at least two elements,
since it contains $x$, and $X_{2} \cap Y_{1}$ is non-empty.
Thus $M$ has a $2$\dash separation, a contradiction.
This shows that $X_{1} \cap Y_{2}$ is $3$\dash separating.

Since $X_{1}$ and $Y_{2} \cup y$ are both $3$\dash separating
the same argument shows that $(X_{1} \cap Y_{2}) \cup y$
is $3$\dash separating.
Since the complement of $X_{2} \cup Y_{1}$ contains both $y$
and at least one element in $X_{1} \cap Y_{2}$, we can also show
that $X_{2} \cap Y_{1}$ and $(X_{2} \cap Y_{1}) \cup x$ are both
$3$\dash separating.
The same argument shows that $X_{2} \cap Y_{2}$ is $3$\dash separating.

Consider~\eqref{separating2}.
The submodularity of the connectivity function shows that
\begin{displaymath}
\la((X_{1} \cap Y_{1}) \cup \{x,\, y\}) + \la(X_{1} \cup Y_{1}) \leq 4.
\end{displaymath}
Thus if $(X_{1} \cap Y_{1}) \cup \{x,\,y\}$ is not
$4$\dash separating then $\la(X_{1} \cup Y_{1}) = 0$.
But this cannot occur as $X_{1} \cup Y_{1}$ is non-empty, and
its complement contains $X_{2} \cap Y_{2}$, which is non-empty.

Next we move to~\eqref{air}.
Since $X_{2} \cup x$ is a flat of $M$ it follows that
$\cl(X_{2})$ does not meet $X_{1}$.
Therefore $\cl(X_{2})$ cannot contain $X_{1} \cap Y_{1}$ or
$X_{1} \cap Y_{2}$.

Suppose that $X_{1} \cap Y_{1}$ is contained in $\cl(Y_{2})$.
Then $Y_{1} - \cl(Y_{2})$ is contained in $X_{2} \cup x$.
However Proposition~\ref{vertcl} says that
\begin{displaymath}
(Y_{1} - \cl(Y_{2}),\, \cl(Y_{2}) - y,\, y)
\end{displaymath}
is a vertical $3$\dash partition of $M$.
Thus $y$ is in the closure of $Y_{1} - \cl(Y_{2})$,
which means that $y \in \cl(X_{2} \cup x)$.
But this is a contradiction as $y \in X_{1}$, and
$X_{2} \cup x$ is a flat of $M$.
The same argument shows that $X_{1} \cap Y_{2}$ is not contained
in $\cl(Y_{1})$.

To prove~\eqref{rank} we suppose that
$\ra((X_{1} \cap Y_{2}) \cup y) \geq 3$.
Consider the partition $(X_{1} \cap Y_{2},\, X_{2} \cup Y_{1},\, y)$
of $E(M)$.
It follows from~\eqref{separating} that
\begin{displaymath}
\la((X_{1} \cap Y_{2}) \cup y) = \la(X_{1} \cap Y_{2}) = 2,
\end{displaymath}
so $\la(X_{2} \cup Y_{1}) = 2$.
Furthermore $y \in \cl(Y_{1})$, so $y$ is in the closure
of $X_{2} \cup Y_{1}$.
Proposition~\ref{guts} shows that $y \in \cl(X_{1} \cap Y_{2})$,
so $\ra(X_{1} \cap Y_{2}) \geq 3$.
Now it is easy to see that
\begin{displaymath}
(X_{1} \cap Y_{2},\, X_{2} \cup Y_{1},\, y)
\end{displaymath}
is a vertical $3$\dash partition of $M$.
However $y \in A \cap X_{1}$ and $X_{1} \cap Y_{2}$ does not
meet $X_{2}$, so we have a contradiction to the fact that
$(X_{1},\, X_{2},\, x)$ is a minimal partition.

We conclude by proving~\eqref{rank2}.
Suppose that $\la((X_{1} \cap Y_{1}) \cup \{x,\, y\}) = 2$.
This implies that $\la(X_{2} \cup Y_{2}) = 2$.
Since $y \in \cl(Y_{2})$ it follows easily that
$\la((X_{1} \cap Y_{1}) \cup x) = 2$.
Consider the partition
\begin{displaymath}
((X_{1} \cap Y_{1}) \cup x,\, X_{2} \cup Y_{2},\, y)
\end{displaymath}
of $E(M)$.
Since $y \in \cl(Y_{2})$ it follows from Proposition~\ref{guts}
that $y$ is in the closure of $(X_{1} \cap Y_{1}) \cup x$.
Thus if $\ra((X_{1} \cap Y_{1}) \cup \{x,\, y\}) \geq 3$ it follows
that $\ra((X_{1} \cap Y_{1}) \cup x) \geq 3$.
In this case
\begin{displaymath}
((X_{1} \cap Y_{1}) \cup x,\, X_{2} \cup Y_{2},\, y)
\end{displaymath}
is  vertical $3$\dash partition of $M$ that violates the
fact that $(X_{1},\, X_{2},\, x)$ is a minimal partition.
This completes the proof of the lemma.
\end{proof}

\begin{prop}
\label{prop8}
Suppose that $(X_{1},\, X_{2},\, x)$ is a minimal partition of the
$3$\dash connected matroid $M$ with respect to the set $A \subseteq E(M)$.
Assume that $(Y_{1},\, Y_{2},\, y)$ is a vertical $3$\dash partition
of $M$ such that $y \in A \cap X_{1}$ and $x \in Y_{1}$.
If $|X_{1} \cap Y_{2}| \geq 2$ then
\begin{displaymath}
\sqcap((X_{1} \cap Y_{1}) \cup \{x,\, y\},\, X_{1} \cap Y_{2}) =
\sqcap((X_{1} \cap Y_{1}) \cup y,\, X_{1} \cap Y_{2}) = 1.
\end{displaymath}
\end{prop}

\begin{proof}
The hypotheses imply that $|E(M)| \geq 4$, so every circuit or cocircuit
of $M$ contains at least three elements.
Let $\pi = \sqcap((X_{1} \cap Y_{1}) \cup \{x,\, y\},\, X_{1} \cap Y_{2})$.
We know from Lemma~\ref{crossetc}~\eqref{rank} that
$\ra(X_{1} \cap Y_{2}) \leq 2$.
Therefore $\pi \leq 2$.
On the other hand, since $|X_{1} \cap Y_{2}| \geq 2$, the fact that
$\ra((X_{1} \cap Y_{2}) \cup y) \leq 2$ implies that
$y \in \cl(X_{1} \cap Y_{2})$.
This in turn implies that $\pi \geq 1$.

Assume that $\pi = 2$.
Then $X_{1} \cap Y_{2} \subseteq \cl((X_{1} \cap Y_{1}) \cup \{x,\, y\})$.
Since $x,\, y \in \cl(Y_{1})$ this means that
$X_{1} \cap Y_{2} \subseteq \cl(Y_{1})$.
But this contradicts~\eqref{air} of Lemma~\ref{crossetc}.
Exactly the same argument shows that
$\sqcap((X_{1} \cap Y_{1}) \cup y,\, X_{1} \cap Y_{2}) = 1$.
\end{proof}

\begin{lemma}
\label{lem1}
Suppose that $(X_{1},\, X_{2},\, x)$ is a minimal partition of the
$3$\dash connected matroid $M$ with respect to the set $A \subseteq E(M)$.
Assume that $(Y_{1},\, Y_{2},\, y)$ is a vertical $3$\dash partition
of $M$ such that $y \in A \cap X_{1}$ and $x \in Y_{1}$.
If $|X_{1} \cap Y_{2}| \geq 2$ then $y \in \cl((X_{1} \cap Y_{1}) \cup x)$.
\end{lemma}

\begin{proof}
The hypotheses imply that every circuit of $M$ contains at least
three elements.
Since $|X_{1} \cap Y_{2}| \geq 2$ it follows from
Lemma~\ref{crossetc}~\eqref{rank} implies that
$y \in \cl(X_{1} \cap Y_{2})$.
We assume that $y \notin \cl((X_{1} \cap Y_{1}) \cup x)$.
Since $X_{1} \cap Y_{1}$ is non-empty by
Lemma~\ref{crossetc}~\eqref{cross} it follows that
$|(X_{1} \cap Y_{1}) \cup x| \geq 2$, so
$\la((X_{1} \cap Y_{1}) \cup x) \geq 2$.
Furthermore $\la((X_{1} \cap Y_{1}) \cup \{x,\, y\}) \leq 3$
by~\eqref{separating2} of Lemma~\ref{crossetc}.
As $y \in \cl(Y_{2})$ we deduce that
\begin{displaymath}
2 \leq \la((X_{1} \cap Y_{1}) \cup x) <
\la((X_{1} \cap Y_{1}) \cup \{x,\, y\}) \leq 3.
\end{displaymath}
Thus $\la((X_{1} \cap Y_{1}) \cup x) = 2$.
Moreover it follows from~\eqref{separating} in
Lemma~\ref{crossetc} that $\la((X_{1} \cap Y_{2}) \cup y) = 2$.
Therefore
\begin{displaymath}
((X_{1} \cap Y_{1}) \cup x,\, (X_{1} \cap Y_{2}) \cup y,\, X_{2})
\end{displaymath}
is an exact $3$\dash partition.

As $x \in \cl(X_{2})$ it follows that
$\sqcap((X_{1} \cap Y_{1}) \cup x,\, X_{2}) \geq 1$.
Now Corollary~\ref{flowers1} implies that
$\sqcap((X_{1} \cap Y_{2}) \cup y,\, X_{2}) \geq 1$.
But~\eqref{air} and~\eqref{rank} of Lemma~\ref{crossetc} imply that
$X_{1} \cap Y_{2} \nsubseteq \cl(X_{2})$ and that
$\ra((X_{1} \cap Y_{2}) \cup y) = 2$.
We deduce that
$\sqcap((X_{1} \cap Y_{2}) \cup y,\, X_{2}) = 1$.
Again using Corollary~\ref{flowers1} we see that
\begin{displaymath}
\sqcap((X_{1} \cap Y_{1}) \cup x,\, (X_{1} \cap Y_{2}) \cup y) = 1.
\end{displaymath}
Proposition~\ref{prop8} tells us that
\begin{displaymath}
\sqcap((X_{1} \cap Y_{1}) \cup \{x,\, y\},\, X_{1} \cap Y_{2}) = 1.
\end{displaymath}
Since $y \in \cl(X_{1} \cap Y_{2})$ we can easily deduce that
$y \in \cl((X_{1} \cap Y_{1}) \cup x)$, contrary to our initial assumption.
\end{proof}

\begin{lemma}
\label{bigstep}
Suppose that $C^{*}$ is a cocircuit of the $3$\dash connected
matroid $M$.
Suppose that $(X_{1},\, X_{2},\, x)$ is a minimal partition
of $M$ with respect to $C^{*}$.
Assume that $\si(M / x_{0})$ is not $3$\dash connected for
any element $x_{0} \in C^{*} \cap X_{1}$.
Let $(Y_{1},\, Y_{2},\, y)$ be a vertical $3$\dash partition
of $M$ such that $y \in C^{*} \cap X_{1}$, and assume that
$x \in Y_{1}$.
Then $|X_{1} \cap Y_{2}| = 1$.
\end{lemma}

\begin{proof}
The hypotheses of the lemma imply that every circuit and cocircuit
of $M$ contains at least three elements.
Let us assume that the lemma fails, so that
$|X_{1} \cap Y_{2}| \geq 2$.
Now~\eqref{rank} of Lemma~\ref{crossetc} implies that
$(X_{1} \cap Y_{2}) \cup y$ contains a triangle of $M$ that contains
$y$.
Since $C^{*}$ meets this triangle in $y$, there must be an element
$z \in X_{1} \cap Y_{2}$ such that $z \in C^{*}$.

By assumption $\si(M / z)$ is not $3$\dash connected so
Proposition~\ref{contr2} implies that there is
vertical $3$\dash partition $(Z_{1}',\, Z_{2}',\, z)$.
Let us assume that $x \in Z_{1}'$.

Suppose that $y \in Z_{i}'$, where $\{i,\, j\} = \{1,\, 2\}$.
Since $\ra((X_{1} \cap Y_{2}) \cup y) = 2$ and $z \in \cl(Z_{i}')$
it follows that $(X_{1} \cap Y_{2}) \cup y \subseteq \cl(Z_{i}')$, as
$y \ne z$ and $z \in X_{1} \cap Y_{2}$.
Let $Z_{i} = Z_{i}' \cup (X_{1} \cap Y_{2}) \cup y$ and let
$Z_{j} = Z_{j}' - Z_{i}$.
Then Proposition~\ref{vertcl} implies that
$(Z_{1},\, Z_{2},\, z)$ is a vertical $3$\dash partition.
Note that $x \in Z_{1}$, whether $i$ is equal to $1$ or $2$.

Suppose that $i = 2$.
Then $(X_{1} \cap Y_{2}) \cup y \subseteq Z_{2} \cup z$.
This means that
$(X_{1} \cap Z_{1}) \cup x \subseteq (X_{1} \cap Y_{1}) \cup \{x,\, y\}$.
Lemma~\ref{lem1} says that $z \in \cl((X_{1} \cap Z_{1}) \cup x)$.
Therefore $z \in \cl((X_{1} \cap Y_{1}) \cup \{x,\, y\})$.
But since $\{y,\, z\}$ spans $(X_{1} \cap Y_{2}) \cup y$ this implies
that $(X_{1} \cap Y_{1}) \cup \{x,\, y\}$ spans
$X_{1} \cap Y_{2}$.
As $x,\, y \in \cl(Y_{1})$ it now follows that $Y_{1}$ spans
$X_{1} \cap Y_{2}$, in contradiction to
Lemma~\ref{crossetc}~\eqref{air}.
Therefore $i = 1$, so $(X_{1} \cap Y_{2}) \cup y \subseteq Z_{1} \cup z$.

We conclude that $X_{1} \cap Z_{2} \subseteq (X_{1} \cap Y_{1}) \cup \{x,\, y\}$.
Suppose that $|X_{1} \cap Z_{2}| \geq 2$.
It follows from~\eqref{rank} of Lemma~\ref{crossetc} that
$\ra((X_{1} \cap Z_{2}) \cup z) = 2$.
Therefore $z$ is in $\cl(X_{1} \cap Z_{2})$, and hence in
$\cl((X_{1} \cap Y_{1}) \cup \{x,\, y\})$.
Exactly as before, we conclude that $Y_{1}$ spans
$X_{1} \cap Y_{2}$, a contradiction.
Therefore $|X_{1} \cap Z_{2}| \leq 1$.

As $\ra(Z_{2}) \geq 3$ we deduce that $|X_{2} \cap Z_{2}| \geq 2$.
But $\la(X_{2} \cap Z_{2}) \leq 2$ by~\eqref{separating} of
Lemma~\ref{crossetc}, so it follows that
$\la(X_{2} \cap Z_{2}) = 2$, and hence
$\la(X_{1} \cup Z_{1}) = 2$.
Now $\la(X_{1} \cup x) + \la(Z_{1} \cup z) = 4$, so
the submodularity of the connectivity function implies that
\begin{displaymath}
\la((X_{1} \cap Z_{1}) \cup \{x,\, z\}) +
\la(X_{1} \cup Z_{1}) \leq 4.
\end{displaymath}
We now conclude that
$\la((X_{1} \cap Z_{1}) \cup \{x,\, z\}) \leq 2$.
It follows from~\eqref{rank2} of Lemma~\ref{crossetc} that
$\ra((X_{1} \cap Z_{1}) \cup \{x,\, z\}) = 2$.

We have already deduced that
$(X_{1} \cap Y_{2}) \cup y \subseteq Z_{1} \cup z$, so
$X_{1} \cap Y_{2} \subseteq (X_{1} \cap Z_{1}) \cup z$.
But $|X_{1} \cap Y_{2}| \geq 2$, and
$\ra((X_{1} \cap Z_{1}) \cup \{x,\, z\}) = 2$.
Therefore $x \in \cl(X_{1} \cap Y_{2})$.
We also know that $y \in \cl(X_{1} \cap Y_{2})$.
Proposition~\ref{prop8} asserts that
\begin{displaymath}
\sqcap((X_{1} \cap Y_{1}) \cup \{x,\, y\},\, X_{1} \cap Y_{2}) = 1.
\end{displaymath}
Since $x,\, y \in \cl(X_{1} \cap Y_{2})$ it follows from
Proposition~\ref{prop7} that $\ra(\{x,\, y\}) \leq 1$,
a contradiction as $M$ is $3$\dash connected.
This completes the proof of the lemma.
\end{proof}

\section{Proof of the main result}
\label{mainproof}

We restate Theorem~\ref{main} here.

\begin{theorem}
\label{bigstep2}
Suppose that $M$ and $N$ are $3$\dash connected matroids such that
$|E(N)| \geq 4$ and $C^{*}$ is a cocircuit of $M$ with the property
that $M / x_{0}$ has an $N$\dash minor for some $x_{0} \in C^{*}$.
Then either:
\begin{enumerate}[(i)]
\item there is an element $x \in C^{*}$ such that $\si(M / x)$ is
$3$\dash connected and has an $N$\dash minor;
\item there is a four-element fan $(x_{1},\, x_{2},\, x_{3},\, x_{4})$
of $M$ such that $x_{1},\, x_{3} \in C^{*}$, and $\si(M / x_{2})$ is
$3$\dash connected with an $N$\dash minor;
\item there is a segment-cosegment pair $(L,\, L^{*})$ such that
$L \subseteq C^{*}$, and $\cl(L) - L$ contains a single element $e$.
In this case $e \notin C^{*}$ and $\si(M / e)$ is $3$\dash connected
with an $N$\dash minor.
Moreover $M / \cl(L)$ is $3$\dash connected with an $N$\dash minor,
and if $x_{i} \in L$ then $M / x_{i}$ is $3$\dash connected up
to a unique spore $(\cl(L) - x_{i},\, y_{i})$; or,
\item there is a segment-cosegment pair $(L,\, L^{*})$ such that
$L$ is a flat and $|L - C^{*}| \leq 1$.
In this case $M / L$ is $3$\dash connected with an $N$\dash minor,
and if $x_{i} \in L$ then $M / x_{i}$ is $3$\dash connected up to
a unique spore $(L - x_{i},\, y_{i})$.
\end{enumerate}
\end{theorem}

\begin{proof}
Assume that $M$ is a counterexample to the theorem.
Let $x_{0}$ be an element of $C^{*}$ such that $N$ is a minor of
$M / x_{0}$.
By hypothesis $\si(M / x_{0})$ is not $3$\dash connected, so
Proposition~\ref{contr2} implies there is a vertical $3$\dash partition
$(Z_{1},\, Z_{2},\, x_{0})$.
It follows easily that $|E(M)| \geq 7$.
By Proposition~\ref{prop5} we will assume, relabeling as necessary,
that $|E(N) \cap Z_{1}| \leq 1$.
Let $Z = Z_{1} - \cl(Z_{2})$.
Lemma~\ref{smallside} implies that $M / e$ has an $N$\dash minor for
every element $e \in Z$, and Lemma~\ref{minimal} implies that there
is a minimal partition $(X_{1},\, X_{2},\, x)$ with respect to
$C^{*}$ such that $x \in C^{*} \cap (Z \cup x_{0})$, and
$X_{1} \subseteq Z$.

Proposition~\ref{biglem} implies that $C^{*}$ has a non-empty
intersection with $X_{1} - \cl(X_{2})$.
If $s \in C^{*} \cap (X_{1} - \cl(X_{2}))$ then $\si(M / s)$ is
not $3$\dash connected by hypothesis.
Therefore there is a vertical $3$\dash partition
$(S_{1},\, S_{2},\, s)$.

\begin{sublemma}
\label{sub1}
Suppose that $s \in C^{*}$ is contained in $X_{1} - \cl(X_{2})$ and
that $(S_{1},\, S_{2},\, s)$ is a vertical $3$\dash partition such
that $x \in S_{1}$.
Then $|X_{1} \cap S_{1}| \geq 2$ and $(X_{1} \cap S_{1}) \cup \{s,\, x\}$
is a segment of $M$.
\end{sublemma}

\begin{proof}
Lemma~\ref{bigstep} tells us that $|X_{1} \cap S_{2}| = 1$.
By Lemma~\ref{crossetc}~\eqref{cross} we know that
$|X_{1} \cap S_{1}| \geq 1$.
Assume that $|X_{1} \cap S_{1}| = 1$.
Then $X_{1}$ contains exactly three elements: the unique element
in $X_{1} \cap S_{2}$, the unique element in $X_{1} \cap S_{1}$,
and $s$.
By the definition of a vertical $3$\dash partition it follows
that $\ra(X_{1}) = 3$ and that $X_{1}$ is a triad of $M$.
As $x \in \cl(X_{1})$ it follows that there is a circuit
$C \subseteq X_{1} \cup x$ that contains $x$.
It cannot be the case that the single element in $X_{1} \cap S_{2}$
is in $C$, for that would imply that
$X_{1} \cap S_{2} \subseteq \cl(S_{1})$,
contradicting Lemma~\ref{crossetc}~\eqref{air}.
As $C$ does not meet the triad $X_{1}$ in a single element
it follows that $(X_{1} \cap S_{1}) \cup \{x,\, s\}$ is a triangle.

If we let $x_{2}$ be the unique element in $X_{1} \cap S_{1}$,
let $x_{4}$ be the unique element in $X_{1} \cap S_{2}$, and
let $x_{1} = x$ and $x_{3} = s$, then
$(x_{1},\, x_{2},\, x_{3},\, x_{4})$ is a four-element fan of $M$.
If $\si(M / x_{2})$ is $3$\dash connected then statement~(ii)
of Theorem~\ref{bigstep2} holds, which is a contradiction as
$M$ is a counterexample to the theorem.
Therefore we will assume that $\si(M / x_{2})$ is not
$3$\dash connected.

Since $\si(M / x_{3})$ is not $3$\dash connected
Theorem~\ref{bixby} asserts that $\co(M \del x_{3})$ is
$3$\dash connected.
Assume that every triad of $M$ that contains $x_{3}$ also
contains $x_{2}$.
Then $\co(M \del x_{3}) \iso M \del x_{3} / x_{2}$.
However $x_{3}$ is contained in a parallel pair in $M / x_{2}$,
so $\si(M / x_{2})$ is obtained from $M \del x_{3} / x_{2}$
by possibly deleting parallel elements.
As $M \del x_{3} / x_{2}$ is $3$\dash connected it follows
that $\si(M / x_{2})$ is $3$\dash connected, contrary to
hypothesis.

Therefore there is a triad $T^{*}$ of $M$ that contains $x_{3}$
but not $x_{2}$.
Now $T^{*}$ cannot meet the triangle $\{x_{1},\, x_{2},\, x_{3}\}$
in exactly one element, and therefore $x_{1} \in T^{*}$.
Let $y_{2}$ be the unique element in $T^{*} - \{x_{1},\, x_{3}\}$.
Since every triad that contains $x_{3}$ must contain either
$x_{1}$ or $x_{2}$, and since both $\{x_{1},\, x_{3}\}$ and
$\{x_{2},\, x_{3}\}$ are contained in triads of $M$ it follows
that $\co(M \del x_{3}) \iso M \del x_{3} / x_{1} / x_{2}$.
Note that $x_{3}$ is a loop of $M / x_{1} / x_{2}$, so
$M \del x_{3} / x_{1} / x_{2} = M / x_{3} / x_{1} / x_{2}$.

As $\si(M / x_{3})$ is not $3$\dash connected there is a
vertical $3$\dash partition $(Z_{1},\, Z_{2},\, x_{3})$
of $M$.
By relabeling as necessary we may assume that $x_{1} \in Z_{2}$.
Hence $x_{2} \in \cl(Z_{2} \cup x_{3})$, so by Proposition~\ref{vertcl}
we may assume that $x_{2} \in Z_{2}$.
Now $(Z_{1},\, Z_{2})$ is an exact $2$\dash separation of
$M / x_{3}$, but $M / x_{3} / x_{1} / x_{2}$ is $3$\dash connected.
By Proposition~\ref{prop1} we see that $Z_{2} - \{x_{1},\, x_{2}\}$
must contain at most one element.
If $Z_{2} = \{x_{1},\, x_{2}\}$ then $\ra(Z_{2}) \leq 2$, a
contradiction.
Therefore $Z_{2} - \{x_{1},\, x_{2}\}$ contains exactly one element.
Let this element be $y_{3}$.
It is easy to see that $Z_{2}$ must be a triad of $M$.

We relabel $x_{4}$ with $y_{1}$.
Let $L = \{x_{1},\, x_{2},\, x_{3}\}$ and let
$L^{*} = \{y_{1},\, y_{2},\, y_{3}\}$.
Now $L$ is a segment of $M$.
Proposition~\ref{prop11} implies $X_{2} \cup x_{1}$ is a
hyperplane, and as $\{x_{1},\, x_{2},\, x_{3}\}$
is a triangle it is easy to see that
$\sqcap(X_{2} \cup x_{1},\, \{x_{2},\, x_{3}\}) = 1$.
If there were some element $e$ in $\cl(L) - L$ then
Proposition~\ref{prop7} would imply that
$\ra(\{e,\, x_{1}\}) \leq 1$, a contradiction.
Therefore $L$ is a flat of $M$.
Moreover $(L - x_{i}) \cup y_{i}$ is a cocircuit of
$M$ for all $i \in \{1,\, 2,\, 3\}$, so $(L,\, L^{*})$ is a
segment-cosegment pair of $M$.

By applying Proposition~\ref{prop10} and Lemma~\ref{crocspore}
we see that $M / L$ is $3$\dash connected, and that
$M / x_{i}$ is $3$\dash connected up to a unique spore
$(L - x_{i},\, y_{i})$ for all $i \in \{1,\, 2,\, 3\}$.
We know that $M / x_{3}$ has an $N$\dash minor.
However $\{x_{1},\, x_{2}\}$ is a parallel pair in $M / x_{3}$,
so $M / x_{3} \del x_{1}$ has an $N$\dash minor.
Furthermore $\{x_{2},\,y_{3}\}$ is a series pair of
$M / x_{3} \del x_{1}$, so $M / x_{3} \del x_{1} / x_{2}$, and hence
$M / L$, has an $N$\dash minor.
Thus statement~(iv) of Theorem~\ref{bigstep2} holds, a contradiction.
We conclude that $|X_{1} \cap S_{1}| \geq 2$.

Since $\la(X_{1} \cup x) = \la(S_{1} \cup s) = 2$
it follows that
\begin{displaymath}
\la((X_{1} \cap S_{1}) \cup \{s,\, x\}) + \la(X_{1} \cup S_{1}) \leq 4.
\end{displaymath}
Suppose that $\la((X_{1} \cap S_{1}) \cup \{s,\, x\}) \geq 3$.
Then $\la(X_{1} \cup S_{1}) \leq 1$, so
$\la(X_{2} \cap S_{2}) \leq 1$.
However, as $|X_{1} \cap S_{2}| = 1$ it follows that
$|X_{2} \cap S_{2}| \geq 2$, so $M$ contains a $2$\dash separation,
a contradiction.
Thus $\la((X_{1} \cap S_{1}) \cup \{s,\, x\}) \leq 2$ and
it follows from Lemma~\ref{crossetc}~\eqref{rank2} that
$(X_{1} \cap S_{1}) \cup \{s,\, x\}$ is a segment.
\end{proof}

\begin{sublemma}
\label{sub3}
The rank of $X_{1} \cup x$ is three.
Moreover, $X_{1}$ is a cocircuit of $M$.
\end{sublemma}

\begin{proof}
Let $s \in C^{*}$ be an element in $X_{1} - \cl(X_{2})$ and
suppose that $(S_{1},\, S_{2},\, s)$ is a vertical $3$\dash partition
such that $x \in S_{1}$.
Then $\ra((X_{1} \cap S_{1}) \cup \{s,\, x\}) = 2$ by~\ref{sub1},
and as $|X_{1} \cap S_{2}| = 1$, Lemma~\ref{crossetc}~\eqref{air}
implies that $\ra(X_{1} \cup x) = 3$.

Proposition~\ref{prop11} asserts that $X_{2} \cup x$ is a flat of $M$,
so $X_{1}$ is a cocircuit.
\end{proof}

\begin{sublemma}
\label{sub5}
Suppose that $y$ and $z$ are elements in $C^{*} \cap X_{1}$, and
$(Y_{1},\, Y_{2},\, y)$ and $(Z_{1},\, Z_{2},\, z)$ are
vertical $3$\dash partitions such that $x \in Y_{1} \cap Z_{1}$.
Then
\begin{displaymath}
|X_{1} \cap Y_{2}| = |X_{1} \cap Z_{2}| = 1 \quad \mbox{and} \quad
X_{1} \cap Y_{2} = X_{1} \cap Z_{2}.
\end{displaymath}
Moreover
\begin{displaymath}
(X_{1} \cap Y_{1}) \cup \{x,\, y\} = (X_{1} \cap Z_{1}) \cup \{x,\, z\}.
\end{displaymath}
\end{sublemma}

\begin{proof}
Let $x'$ be the unique element in $X_{1} \cap Y_{2}$.
From~\ref{sub1} we see that $(X_{1} \cap Y_{1}) \cup \{x,\, y\}$ is
a segment.
The only element of $X_{1}$ not in $(X_{1} \cap Y_{1}) \cup \{x,\, y\}$
is $x'$.
It cannot be the case that $x' \in \cl((X_{1} \cap Y_{1}) \cup \{x,\, y\})$
by Lemma~\ref{crossetc}~\eqref{rank2}.
The same arguments shows that $(X_{1} \cap Z_{1}) \cup \{x,\, z\}$
is a segment, and the only element of $X_{1}$ not in this segment
is $x'$.
Now the result follows easily.
\end{proof}

\begin{sublemma}
\label{sub4}
Let $y \in C^{*}$ be an element in $X_{1}$ and suppose
that $(Y_{1},\, Y_{2},\, y)$ is a vertical
$3$\dash partition such that $x \in Y_{1}$.
Then $|X_{2} \cap Y_{1}| = 1$.
\end{sublemma}

\begin{proof}
We know by~\ref{sub1} that $(X_{1} \cap Y_{1}) \cup \{x,\, y\}$
is a segment.
Let $L' = (X_{1} \cap Y_{1}) \cup \{x,\, y\}$ and let $x'$
be the unique element in $X_{1} \cap Y_{2}$.
Since the complement of $C^{*}$ is a flat of $M$ which
does not contain the segment $L'$ it follows that at most one
element of $L'$ is not contained in $C^{*}$.
As $|X_{1} \cap Y_{1}| \geq 2$ we can find an element
$z \in (X_{1} \cap Y_{1}) \cap C^{*}$.
There must be a vertical $3$\dash partition $(Z_{1},\, Z_{2},\, z)$
such that $x \in Z_{1}$.
From~\ref{sub5} we see that the unique element in $X_{1} \cap Z_{2}$
is $x'$, and that
$(X_{1} \cap Z_{1}) \cup \{x,\, z\} = L'$.

Let $Y_{i}'$ and $Z_{i}'$ denote $X_{2} \cap Y_{i}$ and
$X_{2} \cap Z_{i}$ respectively for $i = 1,\, 2$.
As $(X_{1},\, X_{2},\, x)$ is a minimal partition it follows that
$Y_{i}'$ and $Z_{i}'$ are non-empty for all $i \in \{1,\, 2\}$.
Henceforth we will assume that $|Y_{1}'| > 1$ in order to obtain
a contradiction.

\begin{sublemma}
\label{sub2}
$x \in \cl(Y_{1}')$.
\end{sublemma}

\begin{proof}
We know that $\la(Y_{1}' \cup x) \leq 2$ by
Lemma~\ref{crossetc}~\eqref{separating}.
Since $|Y_{1}'| \geq 2$ it follows that
$\la(Y_{1}' \cup x) = 2$
and hence $\la(X_{1} \cup Y_{2}) = 2$.
Since $x \in \cl(X_{1} \cup Y_{2})$ it follows that
$\la(Y_{1}') = 2$, so Lemma~\ref{guts} implies
that $x \in \cl(Y_{1}')$.
\end{proof}

\begin{sublemma}
\label{sublem2}
Neither $Y_{1}' \cap Z_{1}'$ nor $Y_{2}' \cap Z_{2}'$ is empty.
\end{sublemma}

\begin{proof}
We know from~\ref{sub2} that $x \in \cl(Y_{1}')$.
Since $z \in \cl(Z_{2})$ but
$(X_{1} \cap Z_{1}) \nsubseteq \cl(Z_{2})$, we deduce that
$x \notin \cl(Z_{2})$ as $L'$ is a segment containing both
$x$ and $z$.
Thus $x \notin \cl(Z_{2}' \cup x')$.
Hence $Y_{1}' - Z_{2}' \ne \emptyset$
so $Y_{1}' \cap Z_{1}' \ne \emptyset$.

Note that $z$ is in the closure of
$Z_{2} = Z_{2}' \cup x'$, but
$z \notin \cl(Z_{2}')$ as $X_{1}$ is a cocircuit
by~\ref{sub3}.
This observation means that
$x' \in \cl(Z_{2}' \cup z)$.
However $z \in Y_{1}$, and $x' \notin \cl(Y_{1})$
by Lemma~\ref{crossetc}~\eqref{air}.
Thus $x' \notin \cl(Y_{1}' \cup z)$.
It follows that
$Z_{2}' - Y_{1}' \ne \emptyset$, so
$Z_{2}' \cap Y_{2}' \ne \emptyset$.
\end{proof}

\begin{sublemma}
\label{sublem3}
$(L' \cup (Y_{1}' \cap Z_{1}'), Y_{2} \cup Z_{2})$ is a
$3$\dash separation of $M$.
\end{sublemma}

\begin{proof}
Note that $\la(Y_{2}) = \la(Z_{2}) = 2$, so
$\la(Y_{2} \cap Z_{2}) + \la(Y_{2} \cup Z_{2}) \leq 4$.
From~\ref{sublem2} we see that $Y_{2}' \cap Z_{2}' \ne \emptyset$.
Moreover $x' \in (Y_{2} \cap Z_{2}) - (Y_{2}' \cap Z_{2}')$, which implies
that $|Y_{2} \cap Z_{2}| \geq 2$.
Thus $\la(Y_{2} \cap Z_{2}) \geq 2$, so $\la(Y_{2} \cup Z_{2}) \leq 2$.
As both $L' \cup (Y_{1}' \cap Z_{1}')$ and $Y_{2} \cup Z_{2}$ have cardinality
at least three the claim follows.
\end{proof}

Note that $y,\, z \in \cl(Y_{2} \cup Z_{2})$.
As $y$ and $z$ are contained in the segment $L'$ it follows that
$L' \subseteq \cl(Y_{2} \cup Z_{2})$.
If $|Y_{1}' \cap Z_{1}'| \geq 2$ then it must be the case that
$L' \subseteq \cl(Y_{1}' \cap Z_{1}')$, for otherwise
$(Y_{1}'\cap Z_{1}',\, (Y_{2} \cup Z_{2}) \cup L')$
is a $2$\dash separation of $M$.
But $L' \subseteq \cl(Y_{1}' \cap Z_{1}')$ implies that
$X_{1} \cap Y_{1} \subseteq \cl(X_{2})$, a contradiction.

Therefore $|Y_{1}' \cap Z_{1}'| \leq 1$.
We know from~\ref{sublem2} that $Y_{1}' \cap Z_{1}'$ is not
empty.
Let $e$ be the unique element in $Y_{1}' \cap Z_{1}'$.
Suppose that $e \in \cl(L')$.
As $X_{2} \cup x$ is a hyperplane and $L'$ is a segment
we see that $\sqcap(X_{2} \cup x,\, L' - x) = 1$.
As $e,\, x \in \cl(L' - x)$ it follows from Proposition~\ref{prop7}
that $\ra(\{e,\, x\}) \leq 1$.
We deduce from this contradiction that $e \notin \cl(L')$.

Hence $\ra(L' \cup e) = 3$, so
$\ra(Y_{2} \cup Z_{2}) = \ra(M) - 1$ by~\ref{sublem3}.
Thus the complement of $\cl(Y_{2} \cup Z_{2})$ is a
cocircuit.
However $L' \subseteq \cl(Y_{2} \cup Z_{2})$, so $e$ is
a coloop of $M$, a contradiction.

Our assumption that $|X_{2} \cap Y_{1}| \geq 2$ has lead to an
impossibility.
Since $X_{2} \cap Y_{1}$ is non-empty by
Lemma~\ref{crossetc}~\eqref{cross} we conclude that~\ref{sub4}
is true.
\end{proof}

Now we are in a position to complete the proof of
Theorem~\ref{bigstep2}.
Let $x_{1} = x$, and let $x_{2}$ be some element in
$C^{*} \cap X_{1}$.
There is a vertical $3$\dash partition $(Y_{1}^{2},\, Y_{2}^{2},\, x_{2})$
such that $x_{1} \in Y_{1}^{2}$.
Lemma~\ref{bigstep} tells us that $|X_{1} \cap Y_{2}^{2}| = 1$.
Let $y_{1}$ be the unique element in $X_{1} \cap Y_{2}^{2}$.

We know that $|X_{1} \cap Y_{1}^{2}| \geq 2$ and
$(X_{1} \cap Y_{1}^{2}) \cup \{x_{1},\, x_{2}\}$ is a segment
by~\ref{sub1}.
It follows from Proposition~\ref{prop7}, and the fact that
$(X_{1} \cap Y_{1}^{2}) \cup x_{2}$ is a segment while $X_{2} \cup x_{1}$
is a hyperplane, that $(X_{1} \cap Y_{1}^{2}) \cup \{x_{1},\, x_{2}\}$
is a flat.
The complement of $C^{*}$ can contain at most one element of
$(X_{1} \cap Y_{1}^{2}) \cup \{x_{1},\, x_{2}\}$.
Let $L = C^{*} \cap ((X_{1} \cap Y_{1}^{2}) \cup \{x_{1},\, x_{2}\})$.
Then $\cl(L) = (X_{1} \cap Y_{1}^{2}) \cup \{x_{1},\, x_{2}\}$, and
$\cl(L) - L$ contains at most one element.

Suppose that $L = \{x_{1},\ldots, x_{t}\}$.
We know that $t \geq 3$.
Let $i$ be a member of $\{2,\ldots, t\}$.
As $x_{i} \in C^{*}$ the fact that $M$ is a counterexample
to the theorem means that $\si(M / x_{i})$ is not $3$\dash connected,
so there is a vertical $3$\dash partition $(Y_{1}^{i},\, Y_{2}^{i},\, x_{i})$
such that $x_{1} \in Y_{1}^{i}$.
Then
\begin{displaymath}
(X_{1} \cap Y_{1}^{i}) \cup \{x_{1},\, x_{i}\} =
(X_{1} \cap Y_{1}^{2}) \cup \{x_{1},\, x_{2}\}
\end{displaymath}
by~\ref{sub5}, and~\ref{sub4} implies that there is a unique
element in $X_{2} \cap Y_{1}^{i}$.
Let $y_{i}$ be this element.

Define $L^{*}$ to be $\{y_{1},\ldots, y_{t}\}$.
Note that $L \cap L^{*} = \emptyset$.
We already know that $(\cl(L) - x_{1}) \cup y_{1} = X_{1}$ is
a cocircuit.
Suppose that $i \in \{2,\ldots, t\}$.
Then $(\cl(L) - x_{i}) \cup y_{i}$ is $Y_{1}^{i}$.
As $Y_{1}^{i}$ contains only one element that is not in the segment
$\cl(L)$ it follows that $\ra(Y_{1}^{i}) = 3$.
Thus $\ra(Y_{2}^{i} \cup x_{i}) = r(M) - 1$.
Furthermore $Y_{2}^{i} \cup x_{i}$ is a flat, for otherwise
the complement of $\cl(Y_{2}^{i} \cup x_{i})$ is a cocircuit
of rank at most two, which cannot occur since $M$
is $3$\dash connected.
Hence $(\cl(L) - x_{i}) \cup y_{i}$ is a cocircuit.

We have shown that $(L,\, L^{*})$ is a segment-cosegment pair.
Proposition~\ref{prop10} says that $M / \cl(L)$ is $3$\dash connected.
It is easy to see that the hypotheses of Lemma~\ref{crocspore} are
satisfied, so $M / x_{i}$ is $3$\dash connected up to
the unique spore $(\cl(L) - x_{i},\, y_{i})$, for all
$i \in \{1,\ldots, t\}$.
We know that $M / x_{2}$ has an $N$\dash minor, but as
$\cl(L) - x_{2}$ is a parallel class of $M / x_{2}$ it follows
that $M / x_{2} \del (\cl(L) - \{x_{1},\, x_{2}\})$ has an $N$\dash minor.
Since $\{x_{1},\, y_{2}\}$ is a series pair of
$M / x_{2} \del (\cl(L) - \{x_{1},\, x_{2}\})$ it follows that
$M / x_{2} \del (\cl(L) - \{x_{1},\, x_{2}\}) / x_{1}$, and hence
$M / \cl(L)$, has an $N$\dash minor.

Suppose that $|\cl(L) - C^{*}| = 0$.
Then $L = \cl(L)$, and statement~(iv) of Theorem~\ref{bigstep2}
holds.
Therefore we must assume that there is a single element
$e$ in $\cl(L) - L$.
Lemma~\ref{smallside} tells us that $M / e$ has an $N$\dash minor.
If $\si(M / e)$ is $3$\dash connected, then statement~(iii)
holds.
Therefore we must assume $\si(M / e)$ is not $3$\dash connected.

Let $x_{t+1} = e$.
There must be a vertical $3$\dash partition
$(Y_{1}^{t+1},\, Y_{2}^{t+1},\, x_{t+1})$.
We assume that $x_{1} \in Y_{1}^{t+1}$.
Since $\cl(Y_{1}^{t+1})$ contains $x_{1}$ and $x_{t+1}$ it
follows that $\cl(L) \subseteq \cl(Y_{1}^{t+1})$.
By Proposition~\ref{vertcl} we may assume that $Y_{1}^{t+1}$
contains $\cl(L) - x_{t+1} = L$.

As $X_{2} \cup x_{1}$ is a flat it follows that $x_{t+1} \notin \cl(X_{2})$.
However $x_{t+1} \in \cl(Y_{2}^{t+1})$, so
$X_{1} \cap Y_{2}^{t+1} \ne \emptyset$.
We know that $X_{1} = (L \cup \{x_{t+1},\, y_{1}\}) - x_{1}$,
and as $L \subseteq Y_{1}^{t+1}$ it follows that
$X_{1} \cap Y_{2}^{t+1} = \{y_{1}\}$.

Since $x_{t+1} \in \cl(Y_{2}^{t+1})$, there is a circuit
$C_{1} \subseteq Y_{2}^{t+1} \cup x_{t+1}$ such that
$x_{t+1} \in C_{1}$.
But $Y_{1}^{2} = (L \cup \{x_{t+1},\, y_{2}\}) - x_{2}$ is a
cocircuit of $M$ and $C_{1}$ must meet this cocircuit in more
than one element.
The only element of $Y_{1}^{2} - x_{t+1}$ that can be in $C_{1}$ is
$y_{2}$.
Thus $y_{2} \in Y_{2}^{t+1}$.

Since $(X_{1},\, X_{2},\, x)$ is a minimal partition it follows
that $X_{2} \cap Y_{1}^{t+1}$ is non-empty.
Assume that $|X_{2} \cap Y_{1}^{t+1}| \geq 2$.
As $\la(X_{1}) + \la(Y_{2}^{t+1} \cup x_{t+1}) = 4$, it follows
that
\begin{displaymath}
\la((X_{1} \cap Y_{2}^{t+1}) \cup x_{t+1}) +
\la(X_{1} \cup Y_{2}^{t+1}) \leq 4.
\end{displaymath}
Furthermore
$\la(X_{1} \cup x_{1}) + \la(Y_{2}^{t+1} \cup x_{t+1}) = 4$, so
\begin{displaymath}
\la((X_{1} \cap Y_{2}^{t+1}) \cup x_{t+1}) +
\la(X_{1} \cup Y_{2}^{t+1} \cup x_{1}) \leq 4.
\end{displaymath}
As $(X_{1} \cap Y_{2}^{t+1}) \cup x_{t+1} = \{x_{t+1},\,y_{1}\}$
we deduce that
$\la((X_{1} \cap Y_{2}^{t+1}) \cup x_{t+1}) = 2$.
Thus
\begin{equation}
\label{eq1}
\la(X_{1} \cup Y_{2}^{t+1}),\,
\la(X_{1} \cup Y_{2}^{t+1} \cup x_{1}) \leq 2.
\end{equation}
Both of the sets in Equation~\eqref{eq1} contain at least two elements,
and by assumption $|X_{2} \cap Y_{1}^{t+1}| \geq 2$.
Therefore $X_{2} \cap Y_{1}^{t+1}$ and $(X_{2} \cap Y_{1}^{t+1}) \cup x_{1}$
are exactly $3$\dash separating.
Since $x_{1} \in \cl(X_{1})$ we see from Lemma~\ref{guts} that
$x_{1} \in \cl(X_{2} \cap Y_{1}^{t+1})$.
Thus there is a circuit $C_{2} \subseteq (X_{2} \cap Y_{1}^{t+1}) \cup x_{1}$
such that $x_{1} \subseteq C_{2}$.
We have already noted that $Y_{1}^{2}$ is a cocircuit, and as
$x_{1} \in Y_{1}^{2}$ it follows that $|C_{2} \cap Y_{1}^{2}| \geq 2$.
As $C_{2} - x_{1} \subseteq X_{2}$ the only element other than
$x_{1}$ that can be in $C_{2} \cap Y_{1}^{2}$ is $y_{2}$.
Hence $y_{2} \in C_{2} \subseteq Y_{1}^{t+1}$, a contradiction as we have
already deduced that $y_{2} \in Y_{2}^{t+1}$.

We are forced to conclude that $X_{2} \cap Y_{1}^{t+1}$ contains
a unique element.
Let this element be $y_{t+1}$.
Therefore $Y_{1}^{t+1} = L \cup y_{t+1}$.
Thus $\ra(Y_{1}^{t+1}) = 3$, so $\ra(Y_{2}^{t+1}) = \ra(M) - 1$.
If $Y_{2}^{t+1} \cup x_{t+1}$ is not a hyperplane, then the complement
of $\cl(Y_{2}^{t+1} \cup x_{t+1})$ is a cocircuit of rank at most two,
a contradiction.
Therefore $(\cl(L) - x_{t+1}) \cup y_{t+1} = Y_{1}^{t+1}$ is a
cocircuit.

Let $L_{0} = \{x_{1},\ldots, x_{t+1}\}$ and let
$L_{0}^{*} = \{y_{1},\ldots, y_{t+1}\}$.
Note that $L_{0} = \cl(L)$, so $L_{0}$ is a flat.
We have shown that $(L_{0},\, L_{0}^{*})$ is a segment-cosegment
pair.
Moreover, $M / x_{t+1}$ is $3$\dash connected up to a unique
spore $(L_{0} - x_{t+1},\, y_{t+1})$, by Lemma~\ref{crocspore}.
By relabeling $L_{0}$ and $L_{0}^{*}$ as $L$ and $L^{*}$
respectively we see that statement~(iv) of Theorem~\ref{bigstep2}
holds.
Hence $M$ is not a counterexample, and this contradiction completes
the proof of Theorem~\ref{bigstep2}.
\end{proof}

\section{Acknowledgements}
\label{thanks}

We thank Geoff Whittle for suggesting the problem, and for valuable
discussions.

\end{document}